\documentclass[reqno,11pt]{amsart}
\usepackage[utf8]{inputenc}
\usepackage{amsmath}
\usepackage{amsthm}
\usepackage{amssymb}
\usepackage{amsfonts,mathrsfs}
\usepackage{stmaryrd}
\usepackage{amsxtra}  
\usepackage{epsfig}
\usepackage{verbatim}
\usepackage{enumerate}
\usepackage{xcolor}
\usepackage[all]{xy}

\usepackage{mathtools, hyperref}
\usepackage{tikz}
\usetikzlibrary{shapes}
\usetikzlibrary{plotmarks}
\usepackage{bm}

\theoremstyle{plain}
\newtheorem{theorem}[equation]{Theorem}
\newtheorem{proposition}[equation]{Proposition}
\newtheorem{lemma}[equation]{Lemma}
\newtheorem{corollary}[equation]{Corollary}
\newtheorem{definition}[equation]{Definition}

\theoremstyle{remark}
\theoremstyle{remark}
\newtheorem{remark}[equation]{Remark}
\numberwithin{equation}{section}
\newtheorem{example}[equation]{Example}

\newcommand{\abs}[1]{\left\vert#1\right\vert}
\newcommand{\norm}[1]{\left\Vert#1\right\Vert}
\newcommand{\ol}{\overline}
\newcommand{\wh}{\widehat}

\newcommand{\wt}{\widetilde}

\newcommand{\xdownarrow}[1]{%
  {\left\downarrow\vbox to #1{}\right.\kern-\nulldelimiterspace}
}
\makeatletter
\newcommand*{\dashdownarrow}{%
  \mathrel{%
    \mathpalette\dasharrow@vert{-90}%
  }%
}
\newcommand*{\dashuparrow}{%
  \mathrel{%
    \mathpalette\dasharrow@vert{90}%
  }%
}
\newcommand*{\dasharrow@vert}[2]{%
  \sbox0{$#1\vcenter{}$}%
  \sbox2{$#1\dashrightarrow\m@th$}%
  \dimen@=1.2\dimexpr\ht2-\ht0\relax
  \sbox2{\raisebox{-\ht0}{\unhcopy2}}%
  \ht2=\z@
  \dp2=\z@
  \vcenter{\hbox to 2\dimen@{\hfill\rotatebox{#2}{\box2}\hfill}}%
}
\makeatother

\makeatletter
\renewcommand*\env@matrix[1][\arraystretch]{%
  \edef\arraystretch{#1}%
  \hskip -\arraycolsep
  \let\@ifnextchar\new@ifnextchar
  \array{*\c@MaxMatrixCols c}}
\makeatother

\newcommand{\ca}{{\mathcal A}}

\newcommand{\cc}{{\mathcal C}}

\newcommand{\cg}{{\mathcal G}}

\newcommand{\ck}{{\mathcal K}}

\newcommand{\co}{{\mathcal O}}

\newcommand{\cR}{{\mathcal R}}
\newcommand{\cs}{{\mathcal S}}

\newcommand{\cx}{{\mathbb C}}
\newcommand{\zb}{\mathbb Z}
\newcommand{\tb}{\mathbb T}
\newcommand{\rl}{\mathbb R}

\newcommand{\sF}{{\mathscr F}}

\newcommand{\C}{{\mathbb C}}

\newcommand{\h}{{\mathbb H}}

\newcommand{\Q}{{\mathbb Q}}
\newcommand{\R}{{\mathbb R}}

\newcommand{\Z}{{\mathbb Z}}

\begin{document}

\title[Duality and Approximation]{Duality and approximation of Bergman spaces}
\author{D. Chakrabarti , L. D. Edholm  \& J. D. McNeal}
\begin{abstract} Expected duality and approximation properties are shown to fail on Bergman spaces of domains in $\C^n$, via examples. When the domain admits an operator satisfying certain mapping properties, positive duality and approximation results are proved. Such operators are constructed on generalized Hartogs triangles. On a general bounded Reinhardt domain, norm convergence of Laurent series of Bergman functions is shown. This extends a classical result on Hardy spaces of the unit disc. \end{abstract}
\thanks{Research of the first and third authors was supported by National Science Foundation grants. The first author was also supported by a collaboration grant from the Simons Foundation (\# 316632) and an Early Career internal grant from Central Michigan University.}
\address{Department of Mathematics, \newline Central Michigan University, Mount Pleasant, Michigan, USA}
\email{chakr2d@cmich.edu}
\address{Department of Mathematics, \newline University of Michigan, Ann Arbor, Michigan, USA}
\email{edholm@umich.edu}
\address{Department of Mathematics, \newline The Ohio State University, Columbus, Ohio, USA}
\email{mcneal@math.ohio-state.edu}

\maketitle


\section*{Introduction}\label{S:intro}

If $\Omega\subset\C^n$ is a domain and $p>0$, let $A^p(\Omega)$ denote the Bergman space of holomorphic functions $f$ on $\Omega$ such that

\begin{equation*}
\left\|f\right\|^p_{L^p(\Omega)} = \int_\Omega |f|^p\, dV <\infty,
\end{equation*}
where $dV$ denotes Lebesgue measure.
Three basic questions about function theory on $A^p(\Omega)$ motivate our work:

\begin{description}
\item[\rm{(Q1)}] What is the dual space of $A^p(\Omega)$?
\smallskip
\item[\rm{(Q2)}] Can an element in $A^p(\Omega)$ be norm approximated by holomorphic functions with better global behavior?
\smallskip
\item[\rm{(Q3)}] For $g\in L^p(\Omega)$, how does one construct $G\in A^p(\Omega)$ that is nearest to $g$?
\end{description}

\noindent The questions are stated broadly at this point; precise formulations accompany results in the sections below.

At first glance (Q1-3) appear independent -- one objective of the paper is to show the questions are highly interconnected. On planar domains some connections were shown in \cite{Hed02} and \cite{Durenbook}. Our paper grew from the observation that
irregularity of the Bergman projection described in \cite{EdhMcN16b} has surprising consequences concerning (Q1-3). In particular: there are bounded pseudoconvex domains $D\subset\C^2$ such that 

\begin{itemize}
\item[(a)] the dual space of $A^p(D)$ cannot be identified, even quasi-isometrically, with $A^q(D)$ where $\frac 1p +\frac 1q =1$, 
\item[(b)] there are functions in $A^p(D)$,  $p<2$, that cannot be $L^p$-approximated by functions in $A^2(D)$, and 
\item[(c)] the $L^2$-nearest holomorphic function to a general $g\in L^p(D)$ is not in $A^p(D)$. 
\end{itemize}
Note (a) says the expected Riesz duality pattern  $\left(L^p\right)'\sim L^q$ does not extend to Bergman spaces of general pseudoconvex domains in $\C^n$.
The fact that $\left(A^p\right)'\not\sim A^q$ also has a significant refinement: the identification fails for elementary coefficient functionals. The domains $D$ above are Reinhardt and $0\notin D$. If $f(z)=\sum_{\alpha\in\Z^2} a_\alpha z^\alpha$
belongs to $A^p(D)$, it is not difficult to show the map $f\to a_\alpha$ belongs to $A^p(D)'$. The proof of (a) yields that some of these functionals are not represented as an $L^2$ pairing with a holomorphic function. 

The negative examples frame our positive answers to (Q1-3) and are demonstrated in Section \ref{S:failure}. These results are called {\it breakdowns} of the function theory, to indicate a break with expectations coming from previously studied special cases. But since prior results on (Q1-3) for domains in $\C^n$ are sparse, the examples in Section \ref{S:failure} may represent typical phenomena.  

The initial  goal of the paper is to show how $L^p$ mapping properties of operators related to the Bergman projection, $\bm{B} = \bm{B}_\Omega$, give answers to (Q1-3). Let $\bm{P}:L^2(\Omega)\longrightarrow A^2(\Omega)$ be a bounded operator given by an integral formula

\begin{equation}\label{E:generalOp}
\bm{P}f(z)=\int_\Omega P(z,w) f(w)\, dV(w).
\end{equation}
For $p>0$ fixed, consider the conditions

\begin{description}
\item[\rm{(H1)}] $\exists \, C>0$ such that $\left\|\bm{P}f\right\|_p\leq C\|f\|_p\quad\forall f\in L^p(\Omega)$.  ($\bm{P}$ is bounded on $L^p$)
\smallskip
\item[\rm{(H2)}] $\bm{P}h=h\quad\forall h\in A^p(\Omega)$. ($\bm{P}$ reproduces $A^p$)
\end{description}
Properties (H1-2) will also be invoked on the operators  $\left|\bm{P}\right|$ and $\bm{P}^\dagger$ associated to $\bm{P}$, defined in Section \ref{SSS:basic}. 

A general duality result involves these properties. 
For $1<p<\infty$, let $q$ be the conjugate exponent of $p$. Define the conjugate-linear map $ \Phi_p:A^q(\Omega) \to A^p(\Omega)'$
by the relation $ \Phi_p(g)\big(f\big) =\int_\Omega f\ol{g}\, dV$. 

\begin{proposition}\label{P:intro1}
Let $\Omega \subset \C^n$ be a bounded domain. Let $1<p<\infty$ be given and $q$ be conjugate to $p$ . Suppose there exists $\bm{P}$ of the form  \eqref{E:generalOp}  such that (i) $\left|\bm{P}\right|$ satisfies (H1), (ii) $\bm{P}$ satisfies (H2), and (iii) $Ran\left(\bm{P}^\dagger\right)\subset \co(\Omega)$. 

Then $\Phi_p: A^q(\Omega)\longrightarrow A^p(\Omega)'$ is surjective.
\end{proposition}
\noindent A general approximation result also involves properties (H1-2).

\begin{proposition}\label{P:intro2}
Let $\Omega \subset \C^n$ be a domain.  For a given $1<p< 2$, suppose there exists an operator $\bm{P}$ of the form \eqref{E:generalOp}  such that
$\bm{P}$ satisfies (H1) and (H2).

 Then every $f\in A^p(\Omega)$ can be approximated in the $L^p$ norm by a sequence $f_n\in A^2(\Omega)$.
\end{proposition}

The prime example of an operator \eqref{E:generalOp} is $\bm{P}=\bm{B}$. There is no {\it a priori} reason the Bergman projection should satisfy (H1) or (H2) unless $p=2$, but there are many classes of domains where $\bm{B}$ is known to satisfy both properties for all exponents $1 <p<\infty$ -- see \cite{PhoSte77, McNeal89, NagRosSteWai89, McNeal91, McNeal94, McNeal94b, McNSte94, Koenig02}. On many other classes of domains the answer is unknown. However $\bm{B}$ fails to satisfy property (H1) for all $1< p<\infty$ in general. This was recently established in \cite{BarSah12, EdhMcN16, ChaZey16, EdhMcN16b} for some pseudoconvex domains in $\C^2$. It was observed earlier for classes of roughly bounded planar domains in \cite{LanSte04} and noted for a non-pseudoconvex, but smoothly bounded, family of domains even earlier in \cite{Bar84}. It turns out that $\bm{B}$ also fails to satisfy property (H2) in general; see Example \ref{E:FailureRep1}. 

The second goal of the paper is to construct substitute operators relevant to (Q1-3) in cases where $\bm{B}$ does not satisfy (H1) or (H2). In general this goal is inaugural, but it is achieved for the generalized Hartogs triangles studied in \cite{EdhMcN16b}. The results in Section  \ref{S:GenHartogs} yield the following 

\begin{theorem}\label{T:intro3}

Let $\h_{m/n}$,  $\frac mn \in \Q^+$, be given by \eqref{D:genHartogs}.
 For each $p\ge2$, there is an operator $\widetilde{\bm{P}}$ of the form \eqref{E:generalOp} such that
(i) $\left|\widetilde{\bm{P}}\right|$ satisfies (H1), and (ii) $\widetilde{\bm{P}}$ satisfies (H2).

Moreover, $ \widetilde{\bm{P}}g$ is the unique $L^2$-nearest element in $A^p(\h_{m/n})$ to $g\in L^p(\h_{m/n})$.
\end{theorem}

\noindent The operators $\widetilde{\bm{P}}$ are called {\it sub-Bergman projections}: their kernels are given as subseries of the infinite sum \eqref{E:BergmanInfiniteSum} defining the Bergman kernel. This can be done abstractly (see Section \ref{S:Reinhardt}), but the utility of sub-Bergman operators appears when their kernels can be estimated precisely enough to show they create $A^p$ functions. In such cases, these projections are useful beyond the applications to (Q1-3) shown here.

We mention there is a very fertile area in one--dimensional Hardy space theory to which Theorem \ref{T:intro3} relates, often labeled {\it extremal dual problems}. These problems deal with approximating {\it non-holomorphic} functions 
on the unit circle by holomorphic functions on the disc. There are numerous important results in this area -- see \cite{Durenbook}, Ch. 8; \cite{Garnettbook}, Ch. IV; \cite{KoosisBook}, Ch VII.  These results can be compared/contrasted with the positive result of Proposition \ref{T:SubBergmanMinimization} and the breakdown in Example \ref{E:FailureApprox3} below.

Our third main result concerns (Q2) and does not involve the hypotheses (H1-2). If $\cR$ is a bounded Reinhardt domain and $f\in\co(\cR)$, then $f$ has a unique Laurent expansion $f(z)=\sum_{\alpha\in\Z^n} a_\alpha z^\alpha$ converging uniformly on compact subsets of $\cR$. Note summation is indexed by $\Z^n$ since $0\notin\cR$ is possible. Let $S_Nf$ denote the square partial sum of this series; see Section \ref{SS:TruncatedConvergence}. If
$f\in A^p(\cR)$, these rational functions converge in $L^p$:

\begin{theorem}\label{T:intro4}
Let $\cR$ be a bounded Reinhardt domain in $\C^n$, $1<p <\infty$ and $f\in A^p(\cR)$.  

Then 
\begin{equation*} 
\norm{S_N f - f}_{p} \to 0\qquad\text{as }N\to \infty .
\end{equation*}
\end{theorem}
\noindent This result is a several variables extension of a theorem due to Riesz on Hardy spaces of the unit disc; see pages 104--110 in \cite{Garnettbook} for a proof of Riesz's theorem.

Results about (Q1-3) for $1<p<\infty$ on planar discs are known, which guided our investigation. If $U$ is the unit disc in $\C$ and $p=2$, all three questions have elementary answers. For (Q1), the dual space 
$A^2(U)'$ is isometrically isomorphic to $A^2(U)$ itself, since $A^2(U)$ is a Hilbert space; this fact holds on a general $\Omega\subset \C^n$. For (Q2), if $f(z)=\sum_{n=0}^\infty a_nz^n\in A^2(U)$, then
$\left\|\sum_{n=0}^N a_n z^n -f\right\|_{L^2}\longrightarrow 0$ as $N\to \infty$ by a simple application of Parseval's formula. For (Q3), $G=\bm{B}_U (g)$ gives the $L^2$-closest element in $A^2(U)$ to any $g\in L^2(U)$. 
Since $\bm{B}$ is $L^2$ bounded on a general domain per definition, this fact also holds on a general $\Omega$. For exponents $p\neq 2$, still on the disc $U$, results also exist. The proofs of these results crucially use boundedness of the Bergman or Szeg\H o projection on $L^p(U)$. For (Q1), $A^p(U)'$ is quasi-isometrically isomorphic to $A^q(U)$ where $\frac 1p +\frac 1q =1$ and $1<p<\infty$; see \cite{ZahJud64,Axl88}. Thus the dual spaces of $A^p(U)$ mimic the pattern given by Riesz's characterization of the duals of general $L^p$ spaces, except for a quasi-isometric constant. The constant comes from the operator norm of 
$\bm{B}$ acting on $L^p$. There are also characterizations of $A^1(U)'$ and $A^\infty(U)'$, see \cite{DurSch04}. For (Q2), a dilation argument, see e.g. \cite{DurSch04} page 30, shows that polynomials are dense in $A^p(U)$ for all $0<p<\infty$. For $1<p<\infty$, this density is strengthened in \cite{Garnettbook, Zhu91}: if $f(z)=\sum_{n=0}^\infty a_nz^n\in A^p(U)$, 
$$(*)\,\,\,\, \left\|\sum_{n=0}^N a_n z^n -f\right\|_{L^p}\longrightarrow 0, \qquad N\to \infty.$$
While the form of $(*)$ is the same as when $p=2$, its proof is not an elementary truncation argument.  The proofs in 
\cite{Garnettbook, Zhu91} pass through the Hardy spaces $H^p$ to get estimates in $A^p$ and so use boundedness of the Szeg\H o projection on $L^p$, $1<p<\infty$. We point out the Bergman and Szeg\H o projections on $U$ have the same range of $L^p$ boundedness, but also note this is a special coincidence. Finally for (Q3), the fact that $\bm{B}_U$ is bounded on $L^p$ for $1<p<\infty$ shows  $G=\bm{B}(g)$ solves (Q3), if ``nearest'' is interpreted in the $L^2$ sense. Generalizations of these results on $U$ to simply connected domains $\Omega$ can be proved if the Riemann map from $\Omega$ to $U$ is sufficiently well-behaved, though this seems not to appear in print.
For planar domains other than $U$, the only  significant result about (Q1) known to us is \cite{Hed02}: There are certain domains $\Omega$ such that $A^p(\Omega)'$ is not isomorphic to $A^q(\Omega)$, if $p$ lies outside an interval centered at 2. 

In several variables, duality and approximation questions in the spirit of (Q1-2) seem not to have been considered when $p\neq 2$. However significant results about duality in $L^2$-Sobolev spaces $W^2_s(\Omega)$ have been obtained. Results of this type first occur in work on extension of biholomorphic mappings, \cite{BellLigocka, Bell81a}. These results were greatly developed and generalized in \cite{Bell82a, Bell82b, BellBoas84, Straube84, Komatsu84}. Around (Q2), prior results on approximation in $\co(\Omega)$ have concentrated on uniform norm approximation or the Hilbert norms $W^2_s(\Omega)$. Uniform approximation theorems have been derived from integral formulas but require restrictive geometric assumptions on $b\Omega$, see \cite{Henkin69, Kerzman71, Lieb70, FornaessLeeZhang}. The $W^2_s(\Omega)$ results hold more generally. For instance,
if $\Omega$ is a smoothly bounded pseudoconvex domain, \cite{Cat80} shows $f\in\co(\Omega)$ can be approximated by functions in $\co(\Omega)\cap W^2_s(\Omega)$. In \cite{BarFor86} an analogous result on $C^1$ bounded Hartogs domains in $\C^2$ is proved. See also \cite{Straube10}, Corollaries 5.2 and 5.4. Nevertheless, these results fail without boundary smoothness, see \cite{BarFor86}. This fact partially motivates our insertion of Bergman norms in (Q2). Finally, previous work directed at (Q3) has focused on establishing boundedness of the Bergman projection itself on increasingly wider -- but still smoothly bounded -- classes of domains, as mentioned below Proposition \ref{P:intro2} above. We are unaware of any prior work connected to (Q3) using operators other than the Bergman projection.

The results in the paper are arranged by decreasing generality of the underlying domain. In Section \ref{S:general}, $\Omega$ is a domain with no assumptions on its symmetry or boundary geometry. In some instances, 
$\Omega$ is assumed bounded. The arguments in this section are elementary, but the results apply widely and seem new.  Propositions \ref{P:intro1} and \ref{P:intro2} are slightly extended and established as Theorems \ref{T:surjectivity} and \ref{T:easy_approx}, respectively.
In Section \ref{S:Reinhardt}, bounded Reinhardt domains $\cR$ are considered. The Laurent series expansion of a holomorphic function on $\cR$ provides concrete initial candidates for addressing (Q2) via truncation. Calculation of norms of coefficient functionals related to $L^p$-allowable monomials (Proposition \ref{prop-coefficients}) and a principal value computation (Proposition \ref{P:BonA^p}) are the basic preliminary results. The main result is Theorem \ref{thm-schauder}, a relabeling of Theorem \ref{T:intro4} above. Additionally, Proposition \ref{P:intro1} is applied to give a detailed description about duality of 
$A^p$ on Reinhardt domains in Proposition \ref{P:ReinhardtDuals}.

In Section \ref{S:GenHartogs}, (Q1-3) are considered on the generalized Hartogs triangles studied in \cite{Edh16, EdhMcN16, EdhMcN16b}. The extra symmetries of this family of Reinhardt domains allow precise descriptions of $L^p$ allowable monomials, orthogonality relations, and integrability generally.  The main results are Theorem \ref{T:SubBergmanLpMapping} and Proposition \ref{T:TechnicalSubBergmanStatement}, which construct sub-Bergman projections that are $L^p$ bounded on ranges where $\bm{B}$ is not.   These results imply Theorem \ref{T:intro3}.  Precise versions of the earlier duality and approximation results are obtained in Proposition \ref{T:DualityGenHartogs} and Propositions \ref{T:ApproxGenHartogs1},  \ref{T:ApproxGenHartogs2}. Proposition \ref{T:SubBergmanMinimization} solves a minimization problem that answers a version of (Q3).

\section{Breakdown on the Hartogs triangle}\label{S:failure}

The breakdowns of function theory can be seen on the Hartogs triangle using results established later in the paper and in \cite{EdhMcN16b}. The needed results are referenced below, using notation collected in Section \ref{SS:notation}.

The Hartogs triangle is 
\begin{equation}\label{D:Hartogs}
\h := \left\{(z_1, z_2) \in \C^2: |z_1| < |z_2| < 1 \right\}.
\end{equation}
 In \cite{EdhMcN16b} and  \eqref{D:genHartogs} below, $\h$ is denoted $\h_1$ to indicate membership in a family of domains, but that is not needed here. Abbreviate the Bergman projection $\bm{B}_{\h}$ by $\bm{B}$ for the rest of this section.

Since $\h$ is Reinhardt, every $f\in\co\left(\h\right)$ has a unique Laurent expansion, written
$f(z)=\sum a_\alpha z^\alpha$ 
using standard multi-index notation. Since $z_2\neq 0$ on $\h$ but there are points in $\h$ where $z_1=0$, the summation is taken over
the set $\{\alpha = (\alpha_1,\alpha_2) \in \Z^2 : \alpha_1 \ge 0 \}$.
If $f\in A^p(\h)$, results in Section \ref{S:Reinhardt} show the Laurent expansion of $f$ need only be summed over the smaller set of $L^p$-allowable multi-indices, see \eqref{eq-allowable}.  Denote this set of indices $\cs(\h,L^p)$ -- caveat: this set was denoted $\ca^p_1$ in  \cite{EdhMcN16b}. Corollary \ref{cor-aalpha} implies
\begin{equation}\label{E:HartogsLaurentSeries}
f(z)=\sum_{\alpha\in\cs(\h,L^p)} a_\alpha z^\alpha\qquad\text{if } f\in A^p(\h).
\end{equation}

A special case of \cite[Theorem 1.1 and Remark 4.9]{EdhMcN16b} is

\begin{theorem}\label{T:AbsBergmanLpHartogs}
The absolute value of the Bergman projection $\left|\bm{B}\right|$ on $\h$ is  bounded from $L^p\left(\h\right)$ to $A^p\left(\h\right)$ if and only if $p \in \left(\frac 43, 4\right)$.
\end{theorem}

\noindent Cf.  also \cite{ChaZey16, EdhMcN16}.

\subsection{Failure of representation}\label{SS:FailureRep} The dual space $A^p(\h)'$ is not isomorphic  to $A^q(\h)$ for $p\in\left(\frac 43, 2\right)$ and $q$ conjugate to $p$.
 This is  illustrated with the pair $p=\frac 53$ and $q=\frac 52$; the argument works with minor changes for any $p\in\left(\frac 43, 2\right)$.
 
 Before defining a functional on $A^{5/3}(\h)$, a computation is useful:
\begin{example}\label{E:FailureRep1}  The holomorphic function $h(z_1,z_2)= z_2^{-2} (= z_1^0z_2^{-2})$ satisfies

\begin{itemize}
\item[(i)] $h\in A^{5/3}\left(\h\right)$ and $h\notin A^2(\h)$.
\item[(ii)] $\bm{B}h$ is well-defined and $\bm{B}h \equiv 0$.
\end{itemize}
\end{example}

\begin{proof} Inequality (3.3) in  \cite{EdhMcN16b} or Lemma \ref{L:(p,m/n)-allowable MultiIndices and Norm} below shows that $(0,-2)\in \cs\left(\h, L^{5/3}\right)$ and
$(0,-2)\notin \cs\left(\h, L^{2}\right)$. Thus (i) holds.

Since $\frac 53\in \left(\frac 43, 4\right)$, Theorem \ref{T:AbsBergmanLpHartogs} says $|\bm{B}|$ is bounded on $L^{5/3}(\h)$.
It follows from Proposition \ref{P:BonA^p} that $\bm{B} h$ is well-defined and $\bm{B}h\equiv 0$.
\end{proof}

A non-representable functional is now given using the coefficients in \eqref{E:HartogsLaurentSeries}. 
\begin{example}\label{ex:nonrep1}
The coefficient functional
\[ a_{(0,-2)}:A^{5/3}(\h)\to\C\]
assigning to $f\in A^{5/3}(\h)$ the coefficient of $z_2^{-2}$ in its Laurent expansion is bounded on $A^{5/3}(\h)$.  However, there does {\it not} exist $\phi \in A^{5/2}(\h)$ such that
\[ a_{(0,-2)}(f)=  \langle{f, \phi}\rangle_{\h} .\]
\end{example}
\begin{proof} Uniqueness of the Laurent expansion shows the functional $a_{(0,-2)}$ is well-defined.
Boundedness of $a_{(0,-2)}$ follows from Proposition \ref{prop-coefficients}.

To prove non-representability,  let $h(z) = z_2^{-2}\in\co(\h)$ as above. 
Example \ref{E:FailureRep1} says $h\in A^{5/3}\left(\h\right)$  but $h\notin A^{2}\left(\h\right)$.
Since $(0,-2)\notin \cs\left(\h, L^{2}\right)$, Corollary \ref{C:ExtendedOrthogonalityPairing} shows that for all $g\in A^2(\h)$ 
\begin{equation}\label{eq-rep3}
 \left\langle h, g\right\rangle_\h =0.
 \end{equation}
The fact that $a_{(0,-2)}$ cannot be represented by $\left\langle \cdot, \phi\right\rangle_\h$ for some $\phi\in A^{5/2}(\h)$ is now straightforward. Suppose such a representation held. Note $a_{(0,-2)} (h) =1$ by definition. Since $A^{5/2}(\h)\subset A^2(\h)$, \eqref{eq-rep3}   implies $\left\langle h, \phi\right\rangle_\h =0$ for all $ \phi\in A^{5/2}(\h)$, a contradiction.
\end{proof}

\subsection{Failure of approximation on $A^p$}\label{SS:negative_holo_approx} There are functions $f\in A^{5/3}(\h)$ for which no sequence of functions $f_n\in A^2(\h)$ converges to $f$ in the $L^{5/3}$ norm. As in the previous subsection, minor changes in the argument give an analogous result for any $p\in\left(\frac 43, 2\right)$.
\begin{proposition}\label{P:FailureApprox}
$A^2(\h)$ is not dense in $A^{5/3}(\h)$.
\end{proposition}
\begin{proof}Let $a_{(0,-2)} \in A^{5/3}(\h)'$ and $h\in A^{5/3}\left(\h\right) \setminus A^{2}\left(\h\right) $ be as in the previous section. By Corollary~\ref{cor-aalpha}, since  $(0,-2)\notin \cs\left(\h, L^{2}\right)$, $a_{(0,-2)}$ vanishes on the  linear subspace $A^2(\Omega)$ of
$A^{5/3}\left(\h\right) $. If $A^{2}\left(\h\right) $ were dense in $ A^{5/3}\left(\h\right)  $, continuity would imply $a_{(0,-2)}\equiv 0$ on $A^{5/3}\left(\h\right)$. However, $a_{(0,-2)}\left(h\right)=1$, which contradicts this vanishing. 
\end{proof}
In fact a stronger statement is  true: there are functions in $A^{5/3}(\h)$  that cannot be approximated {\it uniformly on compact subsets of} $\,\h$ by functions in $A^2(\h)$. To see this, suppose that  $\{f_n\}$ is a sequence in $ A^2(\h)$  such that $f_n \to h$
uniformly on compact subsets of $\h$. Recall the  Cauchy representation of  a coefficient of a Laurent series:
\[ a_{(0,-2)}(f)=   \frac{1}{(2\pi i)^2} \int_T \frac{f(\zeta)}{\zeta_2^{-2}} \cdot\frac{d\zeta_1}{\zeta_1} \frac{d\zeta_2}{\zeta_2},
\]
 where $T$ is a  torus contained in $\h$, for example $ \{(z_1,z_2): |z_1| = \frac 14, |z_2| = \frac 12 \} \subset \h$. Since $f_n \to h$ uniformly on $T$ as $n\to \infty$, it follows that
   $a_{(0,-2)} (f_n) \to 1= a_{(0,-2)}(h)$ as $n\to \infty$. This is a contradiction,  since Corollary~\ref{cor-aalpha} $a_{(0,-2)} (f_n) =0$ for  each $n$.

\subsection{Failure of approximation on $L^p$} For $p\geq 4$, there are explicit functions $g\in L^p(\h)$ such that $\bm{B}g\notin A^p(\h)$. Note that $L^p(\h)\subset L^2(\h)$ for this range of $p$, so $\bm{B}g$ is well-defined.
As $g\longrightarrow\bm{B}g$ associates the $L^2$-nearest {\it holomorphic} function to a {\it general} $g$, this is a different failure of approximation than in the previous section.

Since Theorem \ref{T:AbsBergmanLpHartogs} says there does not exist $C$ such that $\|\bm{B}f\|_p\leq C\|f\|_p$ for all $f\in L^p$, the uniform boundedness principle implies the existence of such $g$. But the explicit form of
such ``extremal functions'' (though non-unique) is useful for other purposes. The proofs in  \cite{EdhMcN16b, EdhMcN16} actually show

\begin{example}\label{E:FailureApprox3} On $\h$, let $\psi(z_1, z_2)=\bar z_2$. Then $\bm{B}\psi\notin L^p(\h)$ for any $p\geq 4$.
\end{example}

\begin{proof}
The proof of Proposition 5.1 in \cite{EdhMcN16b} shows that $\bm{B}\psi = Cz_2^{-1}$, for a constant $C \neq0$. An elementary computation in polar coordinates (see Lemma \ref{L:(p,m/n)-allowable MultiIndices and Norm} below) shows that 
$z_2^{-1}\notin L^p(\h)$ if $p\geq 4$.
\end{proof}

Since $\psi\in L^\infty(\h)$, thus in $L^p(\h)$ for all $p>0$, Example \ref{E:FailureApprox3} demonstrates the breakdown mentioned above. 
In \cite{ChaZey16}, the range of $\bm{B}$ acting on $L^p(\h)$ for any $p>4$ is identified as a weighted Bergman space.

\section{General domains}\label{S:general}

\subsection{Notation}\label{SS:notation} Recurring notation and terminology is collected for easy reference.

If $\Omega\subset\C^n$, $\co\left(\Omega\right)$ denotes the set of holomorphic functions on $\Omega$.
The ordinary $L^2$ inner product is written
$\left\langle f, g\right\rangle_\Omega =\int_\Omega f\cdot \bar g\, dV$
where $dV$ is Lebesgue measure. 
For $p>0$, let $\|f\|_p=\left(\int_\Omega |f|^p\, dV\right)^{\frac 1p}$ denote the usual $p$-th power integral; when $p\geq 1$ this defines a norm. $L^p\left(\Omega\right)$ is the class of $f$ with $\|f\|_p <\infty$ and powers $p,q$ satisfying $\frac 1p +\frac 1q=1$ are said to be conjugate. 
The Bergman spaces are $A^p(\Omega) =\co\left(\Omega\right)\cap L^p\left(\Omega\right)$.  

The Bergman projection and kernel are denoted 

\begin{equation}\label{E:BergmanProjDef}
\bm{B}_\Omega f(z)=\int_\Omega B_\Omega(z,w) f(w)\, dV(w),\qquad f\in L^2(\Omega).
\end{equation}
If ambiguity is unlikely, $\bm{B}_\Omega$ is shortened to ${\bm B}$. When the integral in \eqref{E:BergmanProjDef} converges, it is taken as the {\it definition} of
$\bm{B}f$, even if $f\notin L^2(\Omega)$. If $\{\phi_{\alpha}\}_{\alpha \in \ca}$ is an orthonormal basis for $A^2(\Omega)$, the Bergman kernel is  

\begin{equation}\label{E:BergmanInfiniteSum}
B_\Omega(z,w) = \sum_{\alpha \in \ca}\phi_{\alpha}(z)\overline{\phi_{\alpha}(w)}.
\end{equation}

A domain $\cR\subset\C^n$ is called Reinhardt if $(z_1, \dots z_n)\in\cR$ implies $\left(e^{i\theta_1} z_1, \dots, e^{i\theta_n} z_n \right)\in\cR$ for all $(\theta_1, \dots ,\theta_n)\in\R^n$.
If $X$ is a normed linear space, $X'$ will denote its dual space, the set of bounded linear maps $X\to\C$. For $\lambda\in X'$, the standard norm 
$\|\lambda\|_{X'}=\sup\left\{\left|\lambda (f)\right|: \|f\|_X=1\right\}$ is used.

Some notational shorthand is used in Section \ref{S:GenHartogs}.  If $D$ and $E$ are functions depending on several parameters, $D\lesssim E$ means there exists a 
constant $K>0$, independent of specified (or clear) parameters, such that $D\leq K\cdot E$. Finally, 
if $x\in\R$, the floor function $\lfloor x\rfloor$ denotes the greatest integer $\leq x$.

\subsubsection{Two auxiliary operators}\label{SSS:basic} Two operators related to $\bm{P}:L^2(\Omega)\longrightarrow A^2(\Omega)$ given by  \eqref{E:generalOp} occur in hypotheses of results below. The operator $\left|\bm{P}\right|$ is defined
\begin{equation}\label{E:absoluteOp}
\left|\bm{P}\right|f(z)=\int_\Omega |P(z,w)| f(w)\, dV(w)
\end{equation}
where $|P(z,w)|$ denotes absolute value. The triangle inequality shows that if  $\left|\bm{P}\right|$ satisfies (H1), then  $\bm{P}$ does as well. The converse does not necessarily hold.
The operator $\bm{P}^\dagger$ is defined
\begin{equation}\label{E:pseudoOp}
\bm{P}^\dagger f(w)=\int_\Omega \overline{P(z,w)} f(z)\, dV(z).
\end{equation}
Note $\left\langle\bm{P}f, g\right\rangle =\left\langle f, \bm{P}^\dagger g\right\rangle$ holds when Fubini's theorem can be applied, so $\bm{P}^\dagger$ is the formal adjoint of $\bm{P}$.

\subsection{Extending the Bergman projection}\label{SS:extending}
If $\Omega\subset\C^n$ is bounded, $L^{t}(\Omega) \subset L^{s}(\Omega)$ for any $1\leq  s <t$. Thus for $p\geq 2$, $f\in L^p(\Omega)$ implies that $\bm{B}f\in A^2(\Omega)$ and is given by the integral \eqref{E:BergmanProjDef}.
To restate a point in Section \ref{SS:notation},
$\int_\Omega B(z,w) f(w) dV(w)$
is taken as the definition of $\bm{B}f$, whenever the integral converges.
For $p<2$ and $f\in L^p(\Omega)$, this integral does not necessarily converge. Even when it converges, directly determining the size of the integral is difficult -- it is therefore desirable to evaluate $\bm{B}f$ as a limit. 

\subsubsection{Boundedness of the kernel} Various hypotheses on $\Omega$ guarantee convergence of \eqref{E:BergmanProjDef}  for 
$f\in L^p(\Omega)$, $p<2$. For example, let $U \subset \C$ be the unit disc and fix $z\in U$. Then for $f\in L^1\left(U\right)$,
\begin{align*}
\left|\int_U B_{U}(z,w)\, f(w)\, dV(w)\right|=& \left|\frac 1\pi\int_U\frac 1{(1-z\bar w)^2}\, f(w)\, dV(w)\right| 
\leq C_z\,\,\int_U|f(w)|\, dV(w) <\infty.
\end{align*}
Here $C_z=\sup_{w\in U} \left| B_U(z,w)\right|<\infty$,  since $z \in U$ is fixed.
This argument works on a $C^{\infty}$ smoothly bounded strongly pseudoconvex domain \cite{Kerzman72} or more generally on a smoothly bounded pseudoconvex  domain of finite type \cite{Bell86}. 

But the argument fails for the domains $\h_{\gamma}$ defined by \eqref{D:genHartogs}.  Consider $\h_k$ for $k\in\Z^+$ to illustrate. Let $B(z,w)= B_{\h_k}(z_1, z_2,w_1, w_2)$ denote the Bergman kernel. Theorem 1.2 of \cite{Edh16} says
\begin{equation}\label{E:B_h_{k}}
B(z,w) = \frac{p_k(z_1\bar w_1)\cdot\left[ \left(z_2\bar w_2\right)^2 +\left(z_1\bar w_1\right)^k\right] +z_2\bar w_2\cdot q_k\left(z_1\bar w_1\right) }{(1-z_2\bar w_2)^{2}(z_2\bar w_2 -z_1^k\bar w_1^{k})^{2}},
\end{equation}
for explicit polynomials $p_k(s), q_k(s)$ of the complex variable $s$. Two crucial facts are that $p_k(0)=0$ and $q_k(0)\neq 0$.  Let $z=(z_1, z_2)\in \h_k$ be a fixed point (note $z_2\neq 0$) and $w_\delta=(0, \delta)$, $\delta >0$, be a point in $\h_k$ on the $z_2$ axis. Then \eqref{E:B_h_{k}} implies
$B\left(z,w_\delta\right) \approx  \frac{1}{\delta}$.
Letting $\delta\to 0$ shows $B(z, \cdot)\notin L^\infty(\h_k)$.

Other arguments are required to show $\bm{B}$ is defined on $L^p$ for $p<2$ on domains like $\h_{m/n}$.
In \cite{EdhMcN16b}, estimates on $\left|B_{m/n}(z,w)\right|$ and a variant of Schur's test show $|\bm{B}|$ is defined (and bounded) on $L^p(\h_{m/n})$ for an interval of $p<2$; see Theorem \ref{T:BergmanLp} below.  

\medskip

\subsubsection{Limits of exhaustions}\label{SS:principal_value} If $\left|\bm{B}\right|$ is bounded on $L^p(\Omega)$, the integral \eqref{E:BergmanProjDef} is finite. Computing $\bm{B}f$ can be done  as a principal value, a consequence of the following fact:

\begin{proposition}\label{P:pv1} Let $\Omega$ be a domain in $\C^n$. Suppose $\bm{P}$ is an operator of the form \eqref{E:generalOp} such that
$|\bm{P}|$ is bounded on $L^p(\Omega)$ for a given $1<p<\infty$. For $t\in (0,1)$, let $\Omega_t\subset\Omega$ such that 
if $t< t'$, then $\Omega_{t'} \subset\Omega_t$, and $\displaystyle{\bigcup_{t\in (0,1)}\Omega_t = \Omega}$.  

Then if $f\in L^p\left(\Omega\right)$, for almost every $z\in\Omega$ 
\begin{equation}\label{eq-pv1}
\bm{P}f(z)=\lim_{t\to 0} \int_{\Omega_t} P(z,w) f(w)\, dV(w).
\end{equation}
\end{proposition}

\begin{proof} Let $f\in L^p(\Omega)$. The hypothesis on $|\bm{P}|$ says

$$\int_\Omega\left\{\left|\int_\Omega \left|P(z,w)\right| |f(w)|\, dV(w)\right|^p\right\}\, dV(z)\leq C\|f\|^p_p.$$
In particular, for a.e. $z\in \Omega$, the quantity $\{\cdot\}$ above is $<\infty$.  Thus $|P(z,\cdot)|\, |f(\cdot)|\in L^1(\Omega)$ for a.e. $z\in\Omega$.

 Let $\chi_t$ be the indicator function of $\Omega_t$. Note
$\left|  \chi_t (w)P(z,w) \right| \abs{f(w)} \leq \abs{P(z,w)}\abs{ f(w)}$ for any $z\in\Omega$.
 Fix $z$ such that $|P(z,\cdot)|\, |f(\cdot)|\in L^1(\Omega)$.  The dominated convergence theorem 
implies  
\begin{align*}
\lim_{t\to 0}  \left\langle P(z,\cdot), \bar f\right\rangle_{\Omega_t} = \lim_{t\to 0}  \left\langle \chi_t\cdot P(z,\cdot), \bar f\right\rangle_{\Omega} =   \left\langle \lim_{t\to 0}\chi_t\cdot P(z,\cdot), \bar f\right\rangle_{\Omega}
= \bm{P}f(z),
\end{align*}
 as claimed.
\end{proof}

\subsection{Consequences of (H1)}\label{SS:consequences}

Two functional analysis results are derived from assumptions about $L^p$ boundedness of the Bergman projection. Conditions (H1) and (H2), defined below \eqref{E:generalOp}, enter the hypotheses and
conclusions respectively.

\subsubsection{ (H2) and density}

\begin{lemma}\label{L:H2density}  Let $\Omega$ be a domain in $\C^n$. Assume
$\bm{B}$ is bounded on $L^p(\Omega)$ for a given $1<p<\infty$. 

The following statements are equivalent:

\begin{itemize}
\item[(i)] $A^2(\Omega)\cap A^p(\Omega)$ is dense in $A^p(\Omega)$.
\smallskip
\item[(ii)] $\bm{B}h = h\quad\forall h\in A^p(\Omega)$.
\end{itemize}
\end{lemma}

\begin{proof} Assume (i). Then for each $h\in A^p(\Omega)$, there is a sequence $\{h_\nu\}\subset A^2(\Omega)\cap A^p(\Omega)$ such that $h_\nu\to h$ in $A^p(\Omega)$. Since 
$\bm{B}$ is assumed continuous on $L^p(\Omega)$, $\bm{B}h_\nu\to \bm{B}h$. However $\bm{B}h_\nu=h_\nu$, since $h_\nu\in A^2(\Omega)$. Thus, $\bm{B}h=h$.

Assume (ii). Let $h\in A^p(\Omega)$. Since $L^2(\Omega)\cap L^p(\Omega)$ is dense in $L^p(\Omega)$, there exist $g_\nu\in L^2(\Omega)\cap L^p(\Omega)$ such that
$g_\nu\to h$ in $L^p$. Set $h_\nu =\bm{B}g_\nu$. Then $h_\nu\in A^2(\Omega)\cap A^p(\Omega)$ and 
$$h_\nu\to \bm{B}h,$$
since $\bm{B}$ is $L^p$ bounded. As $\bm{B}h=h$ by assumption, (i) holds.
\end{proof}

As mentioned in the Introduction, if $\Omega\subset\C^n$ is a smoothly bounded and pseudoconvex, $\co(\Omega)\cap C^\infty\left(\overline\Omega\right)$ is dense in $A^p(\Omega)$ for all $p\in (1,\infty)$, cf. \cite{Cat80}. Thus (i) holds in this case. Note this density fails in Proposition \ref{P:FailureApprox}. Note also that if $p\geq 2$ and $\Omega$ is any {\it bounded} domain, conditions (i) and (ii) are both trivially satisfied. 

\subsubsection{Generalized self-adjointness} The Bergman projection $\bm{B}$ is self-adjoint on $A^2(\Omega)$: $\left\langle\bm{B}f,g\right\rangle = \left\langle f,\bm{B}g\right\rangle$ if $f,g\in L^2(\Omega)$. This does not automatically imply that $\left\langle\bm{B}f,g\right\rangle = \left\langle f,\bm{B}g\right\rangle$ if $f\in L^p(\Omega)$, $g\in L^q(\Omega)$ for general
conjugate exponents $p$ and $q$. 

However this relation holds when $|\bm{B}|$ satisfies (H1), a consequence of the following general result.

\begin{proposition}\label{P:gen_self_adjoint}
Let $\Omega\subset\C^n$ be a domain. Assume there exists an operator $\bm{P}$ of form \eqref{E:generalOp} and that
 $|\bm{P}|$ is bounded on $L^p(\Omega)$ for a given $1<p<\infty$. Let $q$ be conjugate to $p$.

Then

\begin{itemize}
\item[(i)] $\left|\bm{P}^\dagger\right|$ is bounded on $L^q(\Omega)$.
\smallskip
\item[(ii)] $\left\langle\bm{P}f,g\right\rangle= \left\langle f,\bm{P}^\dagger g\right\rangle\qquad\forall\,\, f\in L^p(\Omega), g\in L^q(\Omega)$.
\end{itemize}
\end{proposition}

\begin{proof} Let $f\in L^p(\Omega), g\in L^q(\Omega)$. Tonelli's theorem implies
\begin{align*}
 \left\langle \left|\bm{P}\right| |f|, |g|\right\rangle =\int_{\Omega}\int_{\Omega}\left| P(z,w)\right|\, \left|g(z)\right|\, \left|f(w)\right|\, dV(w)\, dV(z) = \left\langle  |f|, \left|\bm{P}^\dagger\right||g|\right\rangle. 
\end{align*}
H\" older's inequality and boundedness of $|\bm{P}|$ on $L^p$ yield
\begin{align*}
\left\langle  |f|, \left|\bm{P}^\dagger\right||g|\right\rangle= \left\langle \left|\bm{P}\right| |f|, |g|\right\rangle 
\leq C\, \|f\|_p\, \|g\|_q
\end{align*}
Taking the supremum over $\|f\|_p=1$ shows $\left\|\left|\bm{P}^\dagger\right|g\right\|_q\leq C\|g\|_q$ as claimed.

Fubini's theorem now applies to give (ii):
\begin{align*}
\left\langle\bm{P}f,g\right\rangle &= \int_{\Omega}f(w)\left(\int_{\Omega} P(z,w)\overline{g(z)}\, dV(z)\right)\, dV(w) \\
&= \int_{\Omega}f(w)\overline{\left(\int_{\Omega} \overline{P(z,w)} g(z)\, dV(z)\right)}\, dV(w) 
=\left\langle f,\bm{P}^\dagger g\right\rangle.
\end{align*}
\end{proof}

\begin{remark} The Bergman kernel is conjugate symmetric, $\overline{B(z,w)} = B(w,z)$. Thus if
$|\bm{B}|$ is $L^p$ bounded, (ii) says $\left\langle\bm{B}f,g\right\rangle = \left\langle f,\bm{B}g\right\rangle$ for $f\in L^p(\Omega)$,  $g\in L^q(\Omega)$.
\end{remark}

\subsection{Representing $A^p(\Omega)'$ by $A^q(\Omega)$.}\label{SS:GenDuality} 
The sought for representation is through $L^2$ pairing.
For $1<p<\infty$ define the conjugate-linear map 
\begin{equation}\label{eq-phig}
  \Phi_p(g)\big(f\big) =\int_\Omega f\ol{g}\, dV,\qquad g\in A^q,\, f\in A^p.
\end{equation}
 Hölder's inequality implies $\Phi_p$ maps $A^q(\Omega)$ continuously into  $A^p(\Omega)'$.  
 
 The goal is to understand when $\Phi_p$ is surjective. The preliminary results  hold generally.
 
 \subsubsection{General behavior}\label{SSS:GenBehavior}

\begin{proposition}\label{P:GenDuality}
Let $\Omega\subset \cx^n$ be a bounded domain and $1<p<\infty$. 
\begin{enumerate}
\item[(i)] If  $p\leq 2$, then $\Phi_p$ is injective. 
\item[(ii)] If $p\geq 2$, then $\Phi_p$ has dense image in $A^p(\Omega)'$.
\end{enumerate} 
\end{proposition}
\begin{proof} Let $q$ be the conjugate exponent to $p$. 

For part (i),  suppose that $g\in \ker \Phi_p$; note in particular that $g\in A^q(\Omega)$. Since $p\leq 2$, it follows that $p\leq 2\leq q$, which implies $A^q(\Omega)\subset A^p(\Omega)$. Therefore $g\in A^p(\Omega)$ and
$\Phi_p(g)$ can act on $g$:
$$0=\Phi_p(g)\big(g\big)=\norm{g}_{L^2(\Omega)}^2.$$ Thus $g=0$. 

Consider part (ii). Since $p\geq 2$, necessarily $q\leq 2$. By  part (i),
the map $\Phi_q: A^p(\Omega)\to A^q(\Omega)'$ is injective.  Define the transpose $\Phi_q': (A^q(\Omega)')' \to A^p(\Omega)'$ of $\Phi_q$ 
\[ \Phi_q'(\lambda)(f)= \lambda(\Phi_q f), \quad \lambda\in (A^q(\Omega)')',\ \ f \in A^p(\Omega).\]
Since $\Phi_q$ is injective, the transposed map $\Phi_q'$ has dense image; see \cite{LaxBook}. 

$L^q(\Omega)$ is reflexive; since $A^q(\Omega)\subset L^q(\Omega)$ is closed, $A^q(\Omega)$ is also reflexive. Thus 
the evaluation map $\varepsilon: A^q(\Omega)\to (A^q(\Omega)')'$
defined 
\[ \varepsilon(g)\big(\phi\big)= \phi(g), \quad \phi\in A^q(\Omega)', g\in A^q(\Omega), \]
is an isometric isomorphism. Let $\cc: A^p(\Omega)'\to A^p(\Omega)' $ be the conjugation map defined $(\cc\circ \lambda)(g)= \ol{ \lambda(g)}$; $\cc$ is an antilinear isometric isomorphism of $A^p(\Omega)'$ with itself. To complete the proof of part (ii) it suffices to show
\begin{equation}\label{eq-phipfactoring}
\Phi_p = \cc\circ \Phi_q'\circ \varepsilon, 
\end{equation}
since $\varepsilon$ and $\cc$ are isometric isomorphisms and $\Phi_q'$ has dense image. 

For $f\in A^p(\Omega)$, $g\in A^q(\Omega)$, unraveling yields
\begin{align*}
(\cc\circ \Phi_q'\circ \varepsilon)(g)\big(f\big)= \ol{\Phi_q'(\varepsilon(g))\big(f\big)}
= \ol{\varepsilon(g)\big(\Phi_q f\big)}&= \ol{\left(\Phi_q f\right)\big(g\big)}\\ &=\ol{\int_\Omega g\ol{f}dV}=\int_\Omega f \ol{g}dV= \Phi_p(g)\big(f\big),
\end{align*}
which establishes \eqref{eq-phipfactoring}. 
\end{proof}
Proposition \ref{P:GenDuality} shows $\Phi_p$ is generally {\em almost} surjective.  To show
it is actually surjective would require establishing closed range. This is equivalent  to an estimate of the form
\begin{equation*}
\norm{\Phi_p g}_{A^p(\Omega)'} \gtrsim {\rm dist}(g, \ker \Phi_p),
\end{equation*}
for all $g\in A^q(\Omega)$, where $\ker \Phi_p$ denotes the null space of $\Phi_p$. 

The proof of Proposition \ref{P:GenDuality}  yields the following. Representation \eqref{eq-phipfactoring} is used for the second statement.
\begin{corollary}\label{C:GenDuality}
Let $\Omega \subset \C^n$ be a bounded domain. Suppose the map $\Phi_p:A^q(\Omega)\to A^p(\Omega)'$ is surjective for a given $1<p<\infty$. Let $q$ be conjugate to $p$.  

Then there is a natural identification
\begin{equation*} 
A^p(\Omega)' \cong  \frac{A^q(\Omega)}{\ker \Phi_p}.
\end{equation*}
Furthermore, the map
\begin{equation*}
\Phi_q: A^p(\Omega) \to A^q(\Omega)'
\end{equation*}
is injective and has closed range.

\end{corollary} 

\subsubsection{Surjectivity of $\Phi_p$}\label{SSS:surjectivePhi} Surjectivity of $\Phi_p$ follows from existence of an operator satisfying (H1) and (H2) whose formal adjoint maps into $\co(\Omega)$.

\begin{theorem}\label{T:surjectivity}
Let $\Omega \subset \C^n$ be a bounded domain. Let $1<p<\infty$ be given and $q$ be the conjugate exponent of $p$ . 

Suppose there exists $\bm{P}$ of the form  \eqref{E:generalOp}  and $\cg\subseteq A^p(\Omega)$ such that

\begin{itemize}
\item[(i)] $\left|\bm{P}\right|$ is bounded on $L^p(\Omega)$,
\smallskip
\item[(ii)] $\bm{P}F=F\quad\forall F\in\cg$,
\smallskip
\item[(iii)]  $Ran\left(\bm{P}^\dagger\right)\subset A^q(\Omega)$.
\end{itemize}
Then $\Phi_p: A^q(\Omega)\longrightarrow \cg'$ is surjective.
\end{theorem}

\begin{remark} (a) The case $\cg=A^p(\Omega)$ is included in Theorem \ref{T:surjectivity}.
\smallskip

(b) If $\bm{P}=\bm{B}_\Omega$, hypothesis (iii) is a consequence of (i) by Proposition \ref{P:gen_self_adjoint}.
\end{remark}

\begin{proof}
Let $\lambda\in \cg'$. We want to find a $h\in A^q(\Omega)$, such that $\lambda =\Phi_p(h)$. 
Extend $\lambda$ by the Hahn-Banach theorem to a functional on $L^p(\Omega)$, still denoted $\lambda$, with the same norm. Then there is a $g\in L^q(\Omega)$, 
with $\norm{g}_{L^q}=\|{\lambda}\|_{\left(L^p\right)'}$, such that $\lambda (f) = \int_\Omega f\ol{g}\, dV=\langle f,g\rangle$ for all $f\in L^p(\Omega)$.

Let $h=\bm{P}^\dagger g$; by (iii) $h\in A^q(\Omega)$. Then  for $F\in\cg$ 
\begin{align*}
\Phi_p (h)\big(F\big)&=\langle F,h\rangle= \left\langle F, \bm{P}^\dagger g\right\rangle 
=\left\langle \bm{P}F, g\right\rangle =\langle F,g\rangle= \lambda(F).
\end{align*}
The third equality follows from Proposition \ref{P:gen_self_adjoint}, the fourth follows from  (ii).
\end{proof}

An elementary necessary condition for surjectivity of $\Phi_p$ is worth recording.

\begin{proposition}\label{P:necessary_surjective}
Let $\Omega \subset \C^n$ be a bounded domain. Suppose that for some $p$, $1<p< 2$, $A^2(\Omega)\cap A^p(\Omega)$ is not dense in $A^p(\Omega)$. Then
$\Phi_p$ is not surjective.
\end{proposition}

\begin{proof}
Since $\Omega$ is bounded, $A^2(\Omega)\subset A^p(\Omega)$. The hypothesis thus says that $A^2(\Omega)$ is not dense in $A^p(\Omega)$. By the Hahn-Banach theorem, there exists a non-trivial $\psi\in A^p(\Omega)'$ which vanishes on $A^2(\Omega)$. Let $q$ be the conjugate exponent of $p$. Suppose there were a non-trivial function $g\in A^q(\Omega)$ such that
  $\psi(h)=\int_\Omega h\bar g\, dV\,\,\forall h\in A^p(\Omega)$. Since $q>2$, $g\in A^2(\Omega)$ and $\psi$ acts on $g$. But then
  $0=\psi(g) =\int_\Omega |g|^2 \, dV,$
 contradicting the fact $g$ is not identically zero.
\end{proof}

\subsection{Approximation on $A^p(\Omega)$}\label{S:ApproximationAp} Functions in $A^p(\Omega)$, $1<p<2$, can be approximated by functions in $A^2(\Omega)$ if (H1) and (H2) hold. The next result should be compared with
Proposition \ref{P:FailureApprox}.

\begin{theorem}\label{T:easy_approx}
Let $\Omega \subset \C^n$ be a domain.  For a given $1<p< 2$, suppose there exists an operator $\bm{P}$ of the form \eqref{E:generalOp}  and $\cg\subseteq A^p(\Omega)$ such that

\begin{itemize}  
\item[(i)] $\bm{P}$ is bounded on $L^p(\Omega)$.
\smallskip
\item[(ii)] $\bm{P}h=h\quad\forall h\in \cg$.
\end{itemize}

\noindent Then every $f\in \cg$ can be approximated in the $L^p$ norm by a sequence $f_n\in A^2(\Omega)$.

\end{theorem}
\begin{proof}
Since $f\in \cg \subset L^p(\Omega)$, there exists a sequence $\phi_n\in C_c^{\infty}(\Omega)$ such that $\norm{\phi_n - f}_p \to 0$ as $n\to \infty$.  Letting $f_n := \bm{P}\phi_n$, hypotheses (i) and (ii) give
\begin{align*}
\norm{f_n-f}_p = \left\|{\bm{P}(\phi_n - f)}\right\|_p \lesssim \norm{\phi_n - f}_p.
\end{align*}
Since $\bm{P}:L^2(\Omega)\longrightarrow A^2(\Omega)$, the claimed result holds.
\end{proof}

\begin{remark}  $\Omega$ is not assumed to be bounded in Theorem \ref{T:easy_approx}.
\end{remark}


\section{Reinhardt Domains}\label{S:Reinhardt}

Throughout the section, let $\cR\subset\C^n$ be a bounded Reinhardt domain. The monograph \cite{JarPflBook08} contains extensive information about this class of domains.

\subsection{Integration on Reinhardt domains}\label{SS:ReinhardtIntegration}
Denote by $\abs{\cR}$  the subset of $(\rl^+ \cup \{0\})^n$ defined
\begin{equation*}\label{eq-shadow}
\abs{\cR} = \left\{\left( \abs{z_1},\dots, \abs{z_n}\right) :\, z=(z_1,\dots, z_n)\in \cR\right\},
\end{equation*}
and call this set the {\em Reinhardt shadow} of $\cR$.

For $r\in \abs{\cR}$ and $f$ a continuous function on $\cR$, let $f_r$ be the function on the unit torus $\tb^n=\{\abs{z_j}=1, \,\,\text{for } j=1,\dots ,n\}\subset \cx^n$ defined
$f_r\left(e^{i\theta_1},\dots, e^{i\theta_n}\right)= f\left(r_1e^{i\theta_1},\dots, r_n e^{i\theta_n}\right)$.
Abbreviate this relation by 
\[ f_r(e^{i\theta})= f\left(r_1e^{i\theta_1},\dots, r_n e^{i\theta_n}\right),\]
using vector notation on $r$ and $\theta$. Fubini's theorem implies 

\begin{equation}\label{eq-lpnorm}
\norm{f}_{L^p(\cR)}^p =  \int_{\abs{\cR}} \norm{f_r}^p_{L^p(\tb^n)} r_1r_2\dots r_n dr,
\end{equation}
a form of polar coordinate integration on $\cR$.

\subsection{Holomorphic monomials}\label{SS:HoloMonoReinhardt}
For a multi-index $\alpha\in \mathbb{Z}^n$, 
let $e_\alpha$ denote the monomial function of exponent $\alpha$: $e_\alpha(z)= z^\alpha=z_1^{\alpha_1}\cdots z_n^{\alpha_n},\  z\in \cx^n$.
If $f\in \mathcal{O}(\cR)$, then $f$ has a unique Laurent series expansion 
\begin{equation}\label{eq-laurent}
f = \sum_{\alpha\in \zb^n} a_\alpha(f) e_\alpha
\end{equation}
converging uniformly on compact subsets of $\cR$.
The map
\begin{equation} \label{eq-coefficients}
a_\alpha: \mathcal{O}(\cR)\to \cx
\end{equation}
will be called the $\alpha$-th  {\em coefficient functional}. The uniqueness of the Laurent expansion shows the map $a_\alpha$ is well-defined. 
Cauchy's formula shows $a_\alpha$ is continuous in the natural Fréchet topology of $\mathcal{O}(\Omega)$.

\subsection{The coefficient functionals}\label{SS:CoefficientFunctionals}
In this section, expansion \eqref{eq-laurent} of an $f\in A^p(\cR)$ is shown to consist only of monomials in $A^p(\cR)$.  For $1\leq p\leq \infty$, define the set $\cs(\cR,L^p)$ of {\em $L^p$-allowable multi-indices for $\cR$} by 
\begin{equation}\label{eq-allowable}
\cs(\cR,L^p) := \{\alpha \in \zb^n: \ e_\alpha\in A^p(\cR)\}.
\end{equation}
These were defined  in  \cite{EdhMcN16b}. See \cite{Zwonek99, Zwonek00} connecting such sets to measurements on $\log |\cR|$.
Since $\cR$ is bounded, for $p_1< p_2$ it holds that
$ \cs\left(\cR,L^{p_2}\right)\subset \cs\left(\cR,L^{p_1}\right)$.

\begin{proposition}\label{prop-coefficients}
For each $\alpha\in \cs\left(\cR,L^p\right)$ and $1\leq p\leq \infty$, the coefficient functional
  \[ a_\alpha: A^p(\cR)\to \cx\]
  is bounded. Moreover
  $ \norm{{a_\alpha}}_{A^p(\cR)'} = \frac{1}{\norm{e_\alpha}_{L^p(\cR)}}$.
\end{proposition}
\begin{proof}
Let $\tb = \{\abs{z_j}=r_j \ : j = 1,\dots,n\} \subset \cR$ be a torus. For $f\in A^p(\cR)$, Cauchy's formula implies
\begin{align*} a_\alpha(f) = \frac{1}{(2\pi i)^n} \int_{\tb} \frac{f(\zeta)}{\zeta^\alpha} \cdot\frac{d\zeta_1}{\zeta_1}\dots \frac{d\zeta_n}{\zeta_n}=  \frac{1}{(2\pi)^n}\cdot\frac{1}{r^\alpha} \int_{\tb}{f_r(e^{i\theta})}{e ^{-i \langle \alpha, \theta\rangle} }d\theta,
\end{align*}
where $d\theta=d\theta_1 d\theta_2\dots d\theta_n$ is the volume element of the unit torus.  Hölder's inequality implies
\begin{align}\label{eq-a-alpha}
\abs{a_\alpha(f)}& \leq  \frac{1}{(2\pi)^n}\cdot \frac{1}{r^\alpha}\norm{f_r}_{L^p(\tb)} \norm{1}_{L^q(\tb)}
=  \frac{(2\pi)^{-\frac{n}{p}}}{r^\alpha}\norm{f_r}_{L^p(\tb)}.
\end{align}
When $p=\infty$, interpret $(2\pi)^{-\frac{n}{p}}$ as $1$. 
 
For $1\leq p <\infty$, it follows from \eqref{eq-a-alpha} that
\begin{equation*}
\abs{a_\alpha(f)}^p \cdot (2\pi)^n (r^\alpha)^p \leq  \norm{f_r}_{L^p(\tb)}^p.
\end{equation*}
So if  $\alpha\in \cs(\cR,L^p)$, 
\begin{align}
\abs{a_\alpha(f)}^p \cdot   \norm{e_\alpha}^p_{L^p(\cR)}&= \abs{a_\alpha(f)}^p\cdot   (2\pi)^n\int_{\abs{\cR}} (r^\alpha)^p r_1\dots r_n dr\nonumber\\
&\leq \int_{\abs{\cR}} \norm{f_r}^p_{L^p(\tb)}r_1\dots r_n dr = \norm{f}_{L^p(\cR)}^p.\label{eq-lpest}
\end{align}
If $\alpha\in  \mathcal{S}\left(\cR, L^\infty\right)$, the trivial estimate
$\abs{a_\alpha(f)} \leq \inf_{r\in |\Omega|} \frac{\norm{f}_{\infty}}{r^{\alpha}} = \frac{\norm{f}_{\infty}}{\sup_{z \in \Omega} |z^{\alpha}|} =  \frac{\norm{f}_\infty}{\norm{e_\alpha}_\infty}$
holds.
This estimate and \eqref{eq-lpest}  imply that for all $1\leq p \leq \infty$, 
\[ \norm{a_\alpha}_{A^p(\cR)'}\leq  \frac{1}{\norm{e_\alpha}_{L^p(\cR)}}.\]
Since $a_\alpha(e_\alpha)=1 = \frac{\norm{e_\alpha}_{L^p(\cR)}}{\norm{e_\alpha}_{L^p(\cR)}}$, in fact $ \norm{a_\alpha}_{A^p(\cR)'}=\frac{1}{\norm{e_\alpha}_{L^p(\cR)}}$.
\end{proof}
 
Proposition \ref{prop-coefficients} implies the Laurent expansions of functions in $A^p$ only have monomials that belong to $L^p$: 

\begin{corollary}\label{cor-aalpha}
Let $\cR$ be a bounded Reinhardt domain and $1\leq p\leq\infty$. Let $f\in A^p(\cR)$, with Laurent expansion given by \eqref{eq-laurent}. 

Then if $\alpha\not \in \cs\left(\cR,L^p\right)$, $a_\alpha(f)=0$. Thus
$$ f(z) = \sum_{\alpha\in\cs\left(\cR,L^p\right) } a_\alpha(f) e_\alpha(z).$$
\end{corollary}

\begin{proof}Assume that $a_\alpha(f)\not=0$.
Choose a decreasing family of relatively compact Reinhardt domains $\cR_\epsilon\subset\cR$ such that $\cR_\epsilon
\to \cR$ as $\epsilon\searrow 0$.
It follows from Proposition \ref{prop-coefficients} that 
\[ \abs{a_\alpha(f)}^p \norm{e_\alpha}^p_{L^p(\cR_\epsilon)} \leq   \norm{f}_{L^p(\cR_\epsilon)}^p.\] As $\epsilon \to 0$, the right hand side 
tends to $ \norm{f}_{L^p(\cR)}<\infty$, but the left hand side tends to $\infty$, since
$\norm{e_\alpha}_{L^p(\cR_\epsilon)} \to \infty$. This contradiction proves the result.
\end{proof}

\begin{remark} Take $n=1$, let $U^*=\{0<\abs{z}<1\}$ be the punctured disc, and $p=\infty$. Clearly $\cs\left(U^*, L^{\infty}\right)=\mathbb{N}$.
Corollary~\ref{cor-aalpha} thus says every $f\in A^\infty\left(U^*\right)$ is of the form $f(z)= \sum_{n=0}^\infty a_nz^n$,
and consequently $f$ extends holomorphically to the unit disc. This recaptures Riemann's removable singularity theorem.  A similar argument holds on $A^p(U^*)$ for any $p\ge2$.
\end{remark}

\subsection{Norm convergence of Laurent series}\label{SS:TruncatedConvergence}

If $\cR$ is a bounded Reinhardt domain, $f\in A^p(\cR)$ and $p\in [1,\infty]$, Corollary \ref{cor-aalpha} says
\begin{equation}\label{E:ApLaurentSeries}
 f(z) = \sum_{\alpha\in\cs\left(\cR,L^p\right) } a_\alpha(f) e_\alpha(z),
\end{equation}
with uniform convergence on compact subsets of $\cR$. The goal of this section is to show the series also converges in the $A^p$ norm if $p\in (1,\infty)$.

Since the index set of the series is a subset of an $n$-dimensional lattice, a choice of truncation is required. If $\alpha=(\alpha_1, \dots, \alpha_n)\in \mathbb{Z}^n$ is a multi-index, let
$\abs{\alpha}_\infty = \max \{\abs{\alpha_j}, j=1, \dots, n\}.$ For a formal series  $g(z)=\sum_{\alpha\in \mathbb{Z}^n} b_\alpha e_\alpha(z)$ and a positive integer $N$,  let
\begin{equation*}\label{eq-sn}
 S_Ng = \sum_{\abs{\alpha}_\infty \leq N} b_\alpha e_\alpha.
\end{equation*}
Call this the ``square partial sum'' of the series defining $g$.  

For $p=2$, the square partial sums of \eqref{E:ApLaurentSeries} converge in $A^2(\cR)$ for elementary reasons. Orthogonality of $\{e_\alpha\}$ of $\cR$ gives
\begin{equation*}
\left\|S_Nf- f\right\|_2^2=\sum_{ \stackrel{|\alpha|_\infty >N}{\alpha\in\cs\left(\cR,L^2\right)} } \frac{\left|a_\alpha(f)\right|^2}{\|e_\alpha\|_2^2}.
\end{equation*}
This tends to $0$ as $N\to\infty$ if $f\in A^2(\cR)$. Thus $\{e_\alpha\}$ for $\alpha\in\cs\left(\cR,L^2\right)$ is an orthogonal basis for the Hilbert space $A^2(\cR)$.

An analogous result holds for $p\neq 2$:

\begin{theorem}\label{thm-schauder}
Let $\cR$ be a bounded Reinhardt domain in $\C^n$, $1<p <\infty$ and $f\in A^p(\cR)$.  

Then 
\begin{equation*} 
\norm{S_N f - f}_{p} \to 0\qquad\text{as }N\to \infty .
\end{equation*}
\end{theorem}

The proof of Theorem \ref{thm-schauder} is broken into parts.
 
\subsubsection{Reduction and estimate}
The following fact reduces matters to an estimate plus a simpler density result.

\begin{lemma}\label{lem-schauder1}
Let $T_k$, $k=1,2, \dots$, be a sequence of bounded linear operators from a Banach space $X$ to a Banach space $Y$. Suppose that there is a dense subset 
$D$ of $X$, so that  for each $x\in D$,  $T_k x \to 0$ in the norm of $Y$ as $k\to \infty$. Then the following are equivalent
\begin{enumerate}
\item $\lim_{k\to \infty} \norm{T_kx}=0$ for each $x\in X$.
\item there is a $C>0$ such that for each $k$, we have $\norm{T_k}_{\rm op}\leq C$.
\end{enumerate}
\end{lemma}

\begin{proof}  This is a slight generalization of  \cite[Proposition 1]{Zhu91}. Assume (1). Then (2) holds by the uniform boundedness principle.

Assume (2). Fix $x\in X$ and $\epsilon>0$. Since $D$ is dense in $X$, there exists $p\in D$ such that $\norm{x-p}_X< \frac{\epsilon}{2C}$. Therefore
\begin{align*}
\norm{T_k x}_Y&\leq \norm{T_k x - T_k p}_Y+ \norm{T_k p}_Y  < \frac{\epsilon}{2}+  \norm{T_k p}_Y.
 \end{align*} 
Choosing $k$ so large that $\norm{T_k p}_Y< \frac{\epsilon}{2}$ yields (1). 
\end{proof}


The estimate for Theorem \ref{thm-schauder} is

\begin{lemma}\label{lem-boundedness}
Let $\cR$ be a bounded Reinhardt domain. For each $1<p<\infty$, there exists a constant $C_p$ such that 
$$\left\|S_N f\right\|_{p} \leq C_p \|f\|_p \quad\text{for all } N\in\Z^+, \,\, f\in A^p(\cR).$$
\end{lemma}
\begin{proof}
Denote the unit torus by $\tb^n =\left\{z\in \cx^n: \abs{z_j}=1, \text{ for }j=1,\dots, n \right\}$.
If $g$ is a function on $\tb^n$, let $\sigma_N g$ denote the square partial sum of its Fourier series,
 \[\sigma_N g = \sum_{\abs{\nu}_\infty\leq N}\wh{g}(\nu) e^{i\nu\cdot \theta}.\]
A theorem of Riesz, see e.g. \cite[Chapter VII]{SteinWeiss}, says for each $1<p<\infty$, there is a constant $C_p$ such that $\norm{\sigma_N g}_p \leq C_p\|g\|_p$ independently of $N$.

It follows from \eqref{eq-lpnorm} that
\begin{align*}
\norm{S_N f}_{p}^p &= \int_{\abs{\cR}} \norm{\sigma_N f_r}_{L^p(\tb^n)}^p r_1r_2\dots r_n dr\\
&\leq C_p \int_{\abs{\cR}} \norm{f_r}_{L^p(\tb^n)}^p r_1r_2\dots r_n dr = C_p \norm{f}^p_{p}.
\end{align*}
\end{proof}

\subsubsection{Series expansion of functionals}\label{SSS:FunctionalExpansion} The dense set $D$ needed in Lemma \ref{lem-schauder1} is found by duality.
Given a functional $\lambda \in A^p(\cR)'$, consider the finite sum
\begin{equation}\label{E:PartialSumDual}
  S_N'\lambda = \sum_{\abs{\alpha}_\infty \leq N} \lambda(e_\alpha) a_\alpha,
\end{equation}
where $a_\alpha$ are the coefficient functionals in Proposition \ref{prop-coefficients}.

\begin{proposition}\label{prop-dualconv}
For each $\lambda \in A^p(\cR)'$,
\begin{equation}\label{E:PartialFunctionalsConverge}
\left\|S_N' \lambda - \lambda\right\|_{(A^p)'}\to 0\qquad\text{as } N\to \infty. 
\end{equation}
\end{proposition}
\begin{proof} For $f\in A^p(\cR)$ 
\begin{align*}
S_N' \lambda ( f) = \sum_{\abs{\alpha}_\infty \leq N} \lambda(e_\alpha) a_\alpha(f)
= \lambda \left( \sum_{\abs{\alpha}_\infty \leq N} a_\alpha(f) e_\alpha\right)
 = \lambda (S_N f)
\end{align*}
It follows from Lemma~\ref{lem-boundedness}
\[ \abs{S_N'\lambda (f)} = \abs{\lambda (S_Nf)} \leq  C\|\lambda\|_{\left(A^p\right)'}\, \|f\|_p.\]
Thus $\norm {S_N'}_{\rm op} \leq C$
where  $S_N'$ is viewed as an operator on the Banach space $A^p(\cR)'$. 
\medskip

\noindent{\bf Claim:}  The span of  $\left\{a_\alpha: \alpha\in \cs\left(\cR,L^p\right)\right\}$ is 
dense in $A^p(\cR)'$. 
\smallskip

To prove the claim, let $\mu \in (A^p(\cR)')'$ be an element of the double dual of $A^p(\cR)$ such that $\mu(f)=0$
for each $f$ in the span of  $\left\{a_\alpha: \alpha\in \cs\left(\cR,L^p\right)\right\}$. By the Hahn Banach theorem, it suffices to show that $\mu=0$ on $A^p(\cR)'$.

Since $A^p(\cR)$ is closed in $L^p(\cR)$, $A^p(\cR)$ is reflexive. Therefore there exists a $g\in A^p(\cR)$ such that $\mu(f)=f(g)$ for all $f\in A^p(\cR)'$. Taking $f=a_\alpha$, it follows  that $a_\alpha(g)=0$, i.e. the $\alpha$-th coefficient of the Laurent expansion of the  holomorphic function $g$ vanishes for each $\alpha$. This implies $g=0$, which shows $\mu=0$ and establishes the claim.
\smallskip

To complete the proof, in Lemma~\ref{lem-schauder1}  let $X=Y=A^p(\cR)'$, $T_N=S_N' - {\rm id}$ and $D$  be the linear span of $\left\{a_\alpha: \alpha\in \cs\left(\cR,L^p\right)\right\}$. Note that for each element  $\lambda \in D$, there is an $N$ such that $T_\nu \lambda =0$ for $\nu\geq N$. The hypotheses of Lemma~\ref{lem-schauder1} are thus satisfied; the lemma implies \eqref{E:PartialFunctionalsConverge}.
\end{proof}

\subsubsection{Proof of Theorem~\ref{thm-schauder}} In Lemma~\ref{lem-schauder1}, take $X=Y=A^p(\cR)$,
and $T_N=S_N - {\rm id}$. For each Laurent polynomial $p$, note that $T_N p=0$ for large enough $N$. The result will follow from Lemma~\ref{lem-schauder1} provided it is shown that 
$D=:\left\{{\rm Laurent\,\,\, polynomials } \in A^p(\cR)\right\}$ is a dense subspace of $A^p(\cR)$.  

By Corollary~\ref{cor-aalpha}, $D$ is the linear span of $\{e_\alpha: \alpha\in \cs\left(\cR, L^p\right)\}$. To show this last set is dense, suppose $\lambda\in A^p(\cR)'$ satisfies $\lambda(e_\alpha)=0$ for all $\alpha \in \cs\left(\cR, L^p\right)$. Definition \eqref{E:PartialSumDual} shows $S_N'\lambda =0$ for each $N$. However Proposition~\ref{prop-dualconv} implies $\lambda = \lim S_N'\lambda=0$. Thus, the Hahn-Banach theorem implies ${\rm span}\left\{e_\alpha: \alpha\in \cs\left(\cR, L^p\right)\right\}$ is dense in $A^p(\cR)$. \qed

\subsection{Computing the projection  term-by-term}\label{SS:TermByTerm}
If $\Omega\subset\C^n$ is a bounded domain and $p\geq 2$, $\bm{B}h=h$ for all $h\in A^p(\Omega)$ since $A^p(\Omega)\subset A^2(\Omega)$. For $1<p<2$, this generally fails, even if $\bm{B}$ is $L^p$ bounded.  

However on a bounded Reinhardt domain, if $|\bm{B}|$ satisfies (H1) and $h$ is in the form \eqref{E:ApLaurentSeries}, $\bm{B}h$ can be computed merely by discarding monomials.

\begin{proposition}\label{P:BonA^p}
Let $\cR$ be a bounded Reinhardt domain. For given $1<p< 2$, suppose $\abs{\bm{B}}$ is
bounded on $L^p(\cR)$. 
\begin{enumerate}
\item[(i)] If $\gamma\in \cs\left(\cR, L^p\right)\setminus \cs\left(\cR, L^2\right)$,
then   $e_\gamma\in\ker\bm{B}$. 
\item[(ii)] 
If $f\in A^p(\cR)$ has expansion \eqref{E:ApLaurentSeries}, then
\begin{equation*}
\bm{B} f=  \sum_{\alpha\in \cs\left(\cR, L^2\right)} a_{\alpha}(f)e_{\alpha}.
\end{equation*}
\end{enumerate}
The square partial sums of the series in (ii) converge in  $L^p(\cR)$.
\end{proposition}
\begin{proof} To see (i), 
choose a decreasing family $\{\cR_t: 0<t<1\}$ of relatively compact  {\em Reinhardt} subdomains 
of $\cR$ whose union is  $\mathcal{R}$. Then $e_\gamma\in L^2(\cR_t)$. For each $\beta\in \mathcal{S}(\cR, L^2)$, orthogonality implies $\left\langle e_\gamma, e_\beta\right\rangle_{\cR_t}=0$  since $\gamma\notin \mathcal{S}(\cR, L^2)$.  

Let $B(z,w)$ denote the Bergman kernel of $\cR$. Since $B(z,w) = \sum_{\beta \in \mathcal{S}(\cR, L^2)} \frac{e_\beta(z)\ol{e_\beta(w)}}{\norm{e_\beta}_2^2}$, it follows that
\[ \int_{\cR_t} B(z,w) e_\gamma(w)dV(w)=0.\]
Proposition~\ref{P:pv1} thus yields
\begin{equation}\label{eq-egamma}
\bm{B} e_\gamma =0.
 \end{equation}
  
 To see (ii),  let $f\in A^p(\cR)$. From Theorem~\ref{thm-schauder},  $f= \lim S_N f$ with convergence in $L^p(\cR)$. Since $\bm{B}$ is continuous on $L^p(\cR)$, 
\begin{align*}
\bm{B}f &= \lim_{N\to \infty} \bm{B} (S_N f)
= \lim_{N\to \infty} \bm{B}\left (\sum_{\abs{\alpha}_\infty \leq N} a_\alpha(f) e_\alpha\right) =\lim_{N\to \infty} \sum_{\abs{\alpha}_\infty \leq N} a_\alpha(f) \bm{B}( e_\alpha),
\end{align*}
all limits taken in $L^p$. \eqref{eq-egamma} then yields (ii). 
\end{proof}

\begin{remark}\label{R:Bterm}
Proposition \ref{P:BonA^p} does {\it not} assert that $\bm{B} f\in A^2(\cR)$ for general $f\in A^p(\cR)$ when $1<p<2$. Note that when $1<p<2$
$$\sum_{\alpha\in \cs\left(\cR, L^p\right)} a_{\alpha}e_{\alpha}\in A^p(\cR) \not\Rightarrow \sum_{\alpha\in \cs\left(\cR, L^2\right)} a_{\alpha} e_{\alpha}\in A^2(\cR),$$
 though each of the monomials in the right sum is in $A^2(\cR)$.
\end{remark}

\subsection{Sub-Bergman projections}\label{SS:subBergman} Throughout the section, assume $p\geq 2$. If $\Omega\subset \C^n$ is a bounded domain, let
\begin{equation}\label{eq-m2pdef}
G^{2,p}(\Omega) :=  \ol{\rm span}_{A^2(\Omega)} A^p(\Omega).
\end{equation}
$G^{2,p}(\Omega) \subset L^2(\Omega)$ is a closed subspace.
The {\em $L^p$ sub-Bergman projection} is defined as the orthogonal projection
\begin{equation*} \label{D:LpSubBergmanProj}
\widetilde{\bm{B}_{\Omega}^p}: L^2(\Omega)\to G^{2,p}(\Omega).
\end{equation*}
The representing kernel 
 \begin{equation*}\label{eq-subbergmanintegral} \widetilde{{\bm B}^p_\Omega}f  =\int_\Omega \widetilde{B^p_{\Omega}}(z,w)  f(w)dV(w)\end{equation*}
 is the $L^p$ sub-Bergman kernel. Subscripts are dropped when the domain is unambiguous. Since $A^p(\Omega)\subset G^{2,p}(\Omega)$, it follows that $\widetilde{\bm{B}^p}f=f$ $\forall f\in A^p(\Omega)$.
 
On a Reinhardt domain, the sub-Bergman projection assumes a concrete form. 

 \begin{proposition}  \label{prop-subbergman1}
 Let $\cR$ be a bounded Reinhardt domain in $\cx^n$ and $p\geq 2$. Then
 
 \begin{enumerate}
\item[(i)]  $G^{2,p}(\cR) = \ol{\rm span}_{A^2(\cR)}\left\{e_{\alpha}:\, \alpha \in \cs\left(\cR,L^p\right) \right\}.$
\medskip

\item[(ii)] $\widetilde{B^p}(z,w) = \sum_{\alpha \in \cs\left(\cR,L^p\right)} \frac{e_{\alpha}(z)\ol{e_{\alpha}(w)}}{\norm{e_{\alpha}}_2^2}.$
\end{enumerate}
 \end{proposition}
 
 \begin{proof} This follows from Corollary~\ref{cor-aalpha} and Theorem~\ref{thm-schauder}. Since two norms are involved, details are given for clarity.
 Note  that $\ol{\rm span}_{A^p(\cR)}(\sF) \subset \ol{\rm span}_{A^2(\cR)}(\sF)$ for any $\sF\subset A^2(\cR)$, since $p\geq 2$ and 
 $\cR$ is bounded. Let $g\in G^{2,p}(\cR)$ and $\epsilon >0$. Definition  \ref{eq-m2pdef} says there exists $g'\in A^p(\cR)$ such that $\|g-g'\|_2<\epsilon$. Corollary~\ref{cor-aalpha} and Theorem~\ref{thm-schauder} imply there exist
 $g''\in \ol{\rm span}_{A^p(\cR)}\left\{e_{\alpha}:\, \alpha \in \cs\left(\cR,L^p\right) \right\}$ such that $\|g'-g''\|_2\leq C\|g'-g''\|_p <\epsilon$, $C$ depending on the diameter of $\cR$. Thus (i) holds.
 
 For (ii), since $\widetilde{\bm{B}^p}$ orthogonally projects onto $G^{2,p}(\cR)$, it follows that for $f\in A^p(\cR)$
\begin{equation}\label{eq-projection}
\widetilde{{\bm B}^p}f  = \sum_{\alpha\in \mathcal{S}(\Omega, L^p) }\frac{\langle f, e_\alpha\rangle}{\norm{e_\alpha}_2^2} e_\alpha.
\end{equation}
The series converges in $A^2(\cR)$. The kernel representation (ii) now follows as in ordinary  Bergman theory.
 \end{proof}
 
Let $q$ be conjugate to $p$; note $q\leq 2$. Subspaces of $A^p(\cR)$ and $A^q(\cR)$ enter the next result, and also appear in the description of dual spaces in the next section. Generalizing \eqref{eq-m2pdef},
define the subspace of $A^p(\cR)$
\begin{equation}\label{D:Gqp}
G^{q,p}(\cR):= \overline{\rm span}_{A^q(\cR)}\left\{ e_\alpha: \alpha \in  \cs\left(\cR,L^p\right)\right\}.
\end{equation}
Extending Proposition \ref{P:BonA^p} (i), define the subspace of $A^q(\cR)$
\begin{equation}\label{D:Nqp}
N^{q,p}(\cR):= \overline{\rm span}_{A^q(\cR)}\left\{ e_\alpha: \alpha \in \cs\left(\cR,L^q\right) \setminus \cs\left(\cR,L^p\right)\right\}.
\end{equation}
 
$ \widetilde{\bm{B}^p}$ is not necessarily bounded on $L^p(\cR)$. When $ \left|\widetilde{\bm{B}^p}\right|$ is $L^p$ bounded,  the following holds 

\begin{proposition} \label{prop-subbergman2} Let $\cR$ be a bounded Reinhardt domain in $\cx^n$.  Let $p\geq 2$ and $q$ be conjugate to $p$. 
 Suppose  $\left|\widetilde{{\bm B}^p}\right|$ is bounded on $L^p(\cR)$.  
 
 Then
\begin{enumerate}
\item[(i)]  $\widetilde{{\bm B}^p}$ is a projection from $L^p(\cR)$ onto  $A^p(\cR)$.
\item[(ii)]   Let $\widetilde{{\bm B}^p}^\dagger$ be the formal adjoint defined by \eqref{E:pseudoOp}.
Then  $\widetilde{{\bm B}^p}^\dagger$ is bounded on $L^q(\cR)$. For $f\in L^q(\cR)$, 
 \begin{equation}\label{eq-subbergmanadjoint2}
\widetilde{{\bm B}^p}^\dagger f  = \sum_{\alpha\in \mathcal{S}\left(\cR, L^p\right) }\frac{\langle f, e_\alpha\rangle}{\norm{e_\alpha}_2^2} e_\alpha.
\end{equation}
The square partial sums of the series converge in $A^q(\cR)$.
\item[(iii)] Consider $\widetilde{{\bm B}^p}^\dagger$ restricted to $A^q(\cR)$. Then $\ker \widetilde{{\bm B}^p}^\dagger = N^{q,p}(\cR)$, ${\rm ran }\,\, \widetilde{{\bm B}^p}^\dagger = G^{q,p}(\cR)$, and 
 $\widetilde{{\bm B}^p}^\dagger h=h$  $\forall h\in G^{q,p}(\cR)$.
\end{enumerate}
\end{proposition}

\begin{proof}
The proof of (i) follows directly from the definition of $\widetilde{{\bm B}^p_\cR}$ and the fact that the intersection $G^{2,p}(\Omega)\cap L^p(\Omega)= A^p(\Omega)$.

The first statement in (ii) follows from Proposition \ref{P:gen_self_adjoint} (i). Representation \eqref{eq-subbergmanadjoint2} follows from Proposition \ref{prop-subbergman1} (ii). 
Convergence of the series in $A^q(\cR)$ follows from Theorem~\ref{thm-schauder}.

For (iii), let $\alpha\in \mathcal{S}\left(\cR, L^q\right)\setminus \mathcal{S}\left(\cR, L^p\right)$. Then \eqref{eq-subbergmanadjoint2} shows  $\widetilde{{\bm B}^p}^\dagger (e_\alpha)=0$.  On the other hand, if 
$f\in A^q(\cR)\setminus N^{q,p}(\cR)$, the Laurent series expansion of $f$ must contain a nonzero coefficient of a monomial $e_\beta$ with $\beta\in \mathcal{S}\left(\cR, L^p\right)$. Formula \eqref{eq-subbergmanadjoint2} shows $\widetilde{{\bm B}^p}^\dagger (f)\neq 0$. Thus $\ker \widetilde{{\bm B}^p}^\dagger = N^{q,p}(\cR)$. Additionally, \eqref{eq-subbergmanadjoint2} shows that the range of  $\widetilde{{\bm B}^p}^\dagger$ is  the closure of the linear span of the family $\left\{e_\alpha:\mathcal{S}\left(\cR, L^p\right)\right\}$, i.e. the subspace $G^{q,p}(\cR)$.  The fact that  $\widetilde{{\bm B}^p}^\dagger$ restricts to the identity on $G^{q,p}(\cR)$
follows from \eqref{eq-subbergmanadjoint2} as well.
\end{proof}

\subsection{Representation of $A^p(\cR)'$}\label{SS:ReinhardtDuals}

\begin{proposition}\label{P:ReinhardtDuals}
Let $\cR$ be a bounded Reinhardt domain in $\cx^n$. Let $p\geq 2$ and $q$ be conjugate to $p$. 
 Suppose  $\left|\widetilde{{\bm B}^p}\right|$ is bounded on $L^p(\cR)$.  

\begin{enumerate}
\item[(i)] The map $\Phi_p: A^q(\cR)\to A^p(\cR)'$ is surjective and $\ker\Phi_p = N^{q,p}(\cR)$.
\medskip
\item[(ii)] There is an explicit linear homeomorphism of  Banach spaces
\begin{equation}\label{eq-apdual}
 A^p(\cR)'\cong  G^{q,p}(\cR).
 \end{equation}
 
\item[(iii)] There is a topological direct sum representation
\begin{equation}\label{eq-directsum}
A^q(\cR)'=\Phi_q(A^p(\cR))\oplus \ol{\rm span}_{A^q(\cR)'}\left\{a_\alpha: \alpha\in \mathcal{S}\left(\cR, L^q\right)\setminus \mathcal{S}\left(\cR, L^p\right)\right\}.
\end{equation}
\end{enumerate}

\end{proposition}

\begin{proof} Let $\bm{P}=\widetilde{\bm{B}^p_{\cR}}$ for notational economy.

To see (i),  check the hypotheses of Theorem \ref{T:surjectivity}. Hypothesis (i) of Theorem \ref{T:surjectivity}  is satisfied by assumption. Hypothesis (ii) of the same theorem holds since $A^p(\Omega)\subset G^{2,p}(\Omega)$.
Proposition \ref{P:gen_self_adjoint} implies $\bm{P}^\dagger$ is $L^q$ bounded; since the representing kernel of $\bm{P}^\dagger$ is holomorphic in the free variable, hypothesis (iii) is satisfied. Theorem \ref{T:surjectivity} thus
says $\Phi_p$ is surjective. To determine $\ker\Phi_p$, direct computation gives
\[ \Phi_p(e_\alpha)(e_\beta)= \int_\cR e_\beta \ol{e_\alpha}dV= \norm{e_\alpha}_2^2 \delta_{\alpha,\beta},\]
where $\delta_{\alpha,\beta}$ is the Kronecker symbol. Thus $N^{q,p}(\cR)\subset \ker \Phi_p$.  If $f\in A^q(\cR)\setminus N^{q,p}(\cR)$, there exists $\beta\in \mathcal{S}\left(\cR, L^p\right)$ such that in expansion \eqref{E:ApLaurentSeries} $a_\beta\neq 0$. Then 
$\Phi_p(f)(e_\beta)= a_\beta  \norm{e_\beta}_2^2 \neq 0$, showing $\ker \Phi_p =N^{q,p}(\cR)$.

For (ii), first note the direct sum representation
\begin{equation}\label{eq-aqdirectsum}
A^q(\cR)= N^{q,p}(\cR)\oplus G^{q,p}(\cR).
\end{equation}
$N^{q,p}(\cR)\cap G^{q,p}(\cR)=\{0\}$ holds since the sets are spanned by independent sets of monomials. If $f\in A^q(\cR)$, write
\[ f = \left(f- \bm{P}^\dagger f\right) + \bm{P}^\dagger f.\]
Proposition~\ref{prop-subbergman2} (iii) implies $\ker\bm{P}^\dagger=N^{q,p}(\cR)$ and  ${\rm ran }\,\,\bm{P}^\dagger =G^{q,p}(\cR)$. Therefore \eqref{eq-aqdirectsum} holds.
By (i), $\Phi_p:A^q(\cR)\to A^p(\cR)'$ is surjective and $\ker \Phi_p =N^{q,p}(\cR)$. Thus \eqref{eq-aqdirectsum} and Corollary \ref{C:GenDuality} give

\[ A^p(\cR)' \cong  \frac{A^q(\cR)}{\ker \Phi_p} = \frac{A^q(\cR)}{N^{q,p}(\cR)} =G^{q,p}(\cR),\]
as claimed.

For (iii), let $N^{q,p}(\cR)^\circ$ be the annihilator of $N^{q,p}(\cR)$:
$$N^{q,p}(\cR)^\circ=\left\{\lambda\in A^q(\cR)': \lambda (f)=0\quad\forall f\in N^{q,p}(\cR)\right\}.$$
The decomposition \eqref{eq-aqdirectsum} implies a natural isomorphism $N^{q,p}(\cR)^\circ =G^{q,p}(\cR)$. Proposition~\ref{prop-dualconv} implies that $N^{q,p}(\cR)^\circ$ can be identified with $\lambda\in A^q(\cR)'$
of the form
\[ \lambda =\sum_{\alpha\in \mathcal{S}\left(\cR, L^p\right)} c_\alpha a_\alpha\]
for complex $c_\alpha$, the square partial sums of the series converging in $A^q(\cR)'$.
Thus $N^{q,p}(\cR)^\circ = \ol{\rm span}_{A^q(\cR)'}\left\{a_\alpha : \alpha \in \mathcal{S}\left(\cR, L^p\right)\right\}$.
The same analysis shows $$G^{q,p}(\cR)^\circ = \ol{\rm span}_{A^q(\cR)'}\left\{a_\alpha : \alpha \in  \mathcal{S}\left(\cR, L^q\right)\setminus\mathcal{S}\left(\cR, L^p\right)\right\}.$$
 \eqref{eq-aqdirectsum} yields a direct sum decomposition of the dual spaces
\[ A^q(\cR)' = N^{q,p}(\cR)^\circ \oplus G^{q,p}(\cR)^\circ.\]
The final step is to show the identity $\Phi_q(A^p(\cR))= N^{q,p}(\cR)^\circ$. However if $\alpha\in\cs\left(\cR, L^p\right)$, direct computation yields
 \[ a_\alpha = \frac{1}{\norm{e_\alpha}_2^2} \cdot\Phi_q(e_\alpha),\]
which implies the identity, and therefore \eqref{eq-directsum}.
\end{proof}
\subsection{Holomorphic Sobolev spaces}

For a multi-index $\beta\in \mathbb{N}^n$, let $\partial^\beta$ denote the partial differential operator $\partial^\beta = \partial^{\abs{\beta}}/{\partial z_1^{\beta_1}\dots \partial z_n^{\beta_n}}$ on $\cx^n$. 
Note that 
$\partial^\beta e_\alpha =C(\alpha,\beta) \cdot e_{\alpha-\beta}$,
where
\[ C(\alpha,\beta)= \begin{cases}  0 & \text{if there is a $j$ such that $\beta_j>\alpha_j \geq 0$} \\
\prod_{j=1}^n \prod_{\ell=0}^{\beta_j-1} (\alpha_j-\ell),&\text{otherwise.}\end{cases}\]
The empty product, in the case $\beta_j=0$, is defined to be 1. 
Let $1\leq p \leq \infty$ and $k\in\mathbb{N}$. Consider the holomorphic Sobolev spaces defined 
\begin{equation*} 
A^p_k(\Omega)=\left\{f\in \mathcal{O}(\Omega):\, \partial^\beta f \in A^p(\Omega) \text{ for each } \beta\in \mathbb{N}^n, \text{ with } \abs{\beta}\leq k\right\}.
\end{equation*}
Note that $A^p_0(\Omega)=A^p(\Omega)$.

If $\cR\subset \cx^n$ is  Reinhardt, let $\mathcal{S}\left(\cR, A^p_k\right)$ denote the set of
$\alpha\in \mathbb{Z}^n$ such that $e_\alpha\in A^p_k(\cR)$. The following generalization of Corollary~\ref{cor-aalpha} holds.

\begin{proposition}\label{prop-sobolev} Let $\cR$ be a bounded Reinhardt domain, $1\leq p\leq\infty$, and $k\in\mathbb{N}$.  Let  $f\in A^p_k(\cR)$, with
Laurent expansion given by \eqref{eq-laurent}. 

Then if $\alpha\notin \cs\left(\cR,A^p_k\right)$,
 $a_\alpha(f)=0$.
\end{proposition}
\begin{proof}
The case $k=0$ is Corollary~\ref{cor-aalpha}. For $k\geq 1$, let $f\in A^p_k(\cR)$, and $\alpha\not \in \cs\left(\cR,A^p_k\right)$. Thus there is a $\beta\in\mathbb{N}^n$, such that $\abs{\beta}\leq k$, and $\partial^\beta e_\alpha\not \in A^p(\cR)$. Since  $\partial^\beta e_\alpha=C(\alpha,\beta) e_{\alpha-\beta}$, this implies two facts: (i) $C(\alpha,\beta)\neq 0$ and (ii) $e_{\alpha-\beta}\not \in A^p(\Omega)$.

As $f\in A^p_k(\Omega)$, necessarily $\partial^\beta f \in A^p(\Omega)$. Corollary~\ref{cor-aalpha} and fact (ii) imply a third fact: (iii) $a_{\alpha-\beta}\left(\partial^\beta f\right)=0$.
 Differentiating the Laurent 
expansion \eqref{eq-laurent} gives
\[ \partial^\beta f = \sum_{\gamma\in \mathbb{Z}^n}a_\gamma(f) \partial^\beta e_\gamma
= \sum_{\gamma\in \mathbb{Z}^n}a_\gamma(f)  C(\gamma,\beta) e_{\gamma-\beta}.\]
Comparing coefficients yields $a_{\alpha-\beta}(\partial^\beta f)=C(\alpha,\beta) \cdot a_\alpha(f)$. This implies $a_\alpha(f)=0$, by facts (i) and (iii).
\end{proof}

\begin{corollary}\label{cor-sibony} 
Let $\cR$ be a Reinhardt domain in $\C^n$ and $\widetilde{\cR}$ be the smallest complete Reinhardt domain containing $\cR$.  Suppose that for some $1\leq p \leq \infty$,
\begin{equation*} 
\bigcap_{k=0}^\infty \cs\left(\cR, A^p_k\right)= \mathbb{N}^n.
\end{equation*}
Then every $f\in \mathcal{C}^\infty(\ol{\cR})\cap \mathcal{O}(\cR)$ extends holomorphically to $\wt{\cR}$.
\end{corollary}

\begin{proof} Since $f\in A^p_k(\cR)$ for each $k$, Proposition~\ref{prop-sobolev} says
the Laurent series of $f$ contains no monomials with negative exponents. The Laurent series thus reduces to a Taylor series. 
The series necessarily converges in some neighborhood of zero and defines an analytic continuation of $f$ to $\wt{\cR}$.
\end{proof}

\begin{example}\label{Ex:SibonyHartogs}
Consider the Hartogs triangle $\h$.  It is a classical fact that any function holomorphic {\em in a neighborhood} of $\overline{\h}$ extends holomorphically to the bidisc.  However a stronger result is true: any $f\in \mathcal{C}^{\infty}(\ol{\h})\cap \co(\h)$ extends to a holomorphic function on the bidisc.

To see this, write
\begin{equation*}
e_\alpha(z)= z_1^{\alpha_1}z_2^{\alpha_2} = \left(\frac{z_1}{z_2}\right)^{\alpha_1} z_2^{\alpha_1+\alpha_2},
\end{equation*}
and recall that $|z_1|<|z_2|$ if $(z_1, z_2)\in\h$. It follows that

\begin{equation}\label{E:H1BoundedMonomials}
\cs\left(\h, A_0^\infty\right)=\left\{\left(\alpha_1, \alpha_2\right): \alpha_1\geq 0,\ \alpha_1+\alpha_2 \geq 0\right\}.
\end{equation}
On the other hand, since $\partial^\beta e_\alpha=C(\alpha, \beta)e_{\alpha-\beta}$, $\partial^\beta e_\alpha\in A_0^\infty(\h)$ if $\alpha_1\geq \beta_1$ and $\alpha_1+\alpha_2\geq \beta_1+\beta_2$. Therefore,

\begin{equation}\label{E:H1BoundedDerivativeMonomials}
\cs(\h, A^\infty_k) = \left\{\left(\alpha_1, \alpha_2\right):\alpha_1\geq 0,\alpha_2 \geq 0\right\} \cup \{(\alpha_1, \alpha_2):\alpha_1\geq k,\ \alpha_1+\alpha_2 \geq k\}.
\end{equation}

The situation is illustrated below, in the fourth quadrant of the lattice point diagram of $\h$.  The lattice points $\alpha=(\alpha_1,\alpha_2)$ on and above the line indexed by $A_k^{\infty}$ correspond to monomials $e_{\alpha}(z)=z^{\alpha}$ with $\alpha\in \cs\left(\h,A_k^{\infty}\right)$.  Differentiation with respect to $z_1$ (resp. $z_2$) is denoted by $\partial_1$ (resp. $\partial_2$), and is represented (up to a constant multiple) by a shift left (resp. a shift down).  
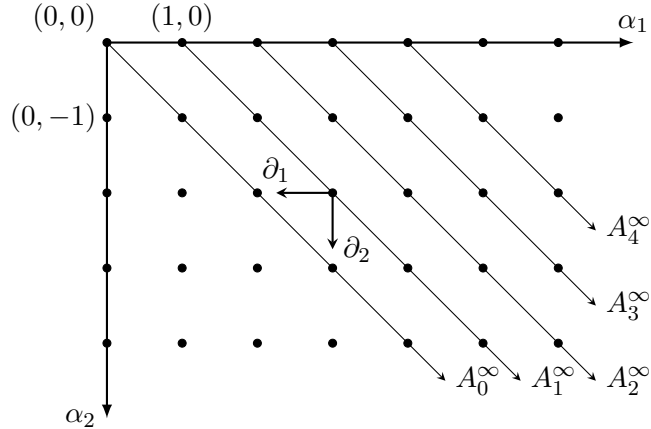
{
\begin{center}
\begin{tikzpicture}

\centering

\draw[-{latex}, thick] (0,0) -- (7,0) node[anchor=south] {$\alpha_1$};
\draw[-{latex}, thick] (0,0) -- (0,-5) node[anchor=east] {$\alpha_2$};

\draw[-stealth, thin] (0,0) -- (4.5,-4.5)  node[anchor=west] {$A_0^{\infty}$};
\draw[-stealth, thin] (1,0) -- (5.5,-4.5) node[anchor=west] {$A^{\infty}_1$};
\draw[-stealth, thin] (2,0) -- (6.5,-4.5) node[anchor=west] {$A^{\infty}_2$};
\draw[-stealth, thin] (3,0) -- (6.5,-3.5) node[anchor=west] {$A^{\infty}_3$};
\draw[-stealth, thin] (4,0) -- (6.5,-2.5) node[anchor=west] {$A^{\infty}_4$};

\draw[-stealth, thick] (3,-2) -- (2.25,-2) node[anchor=south] {$\partial_1$};
\draw[-stealth, thick] (3,-2) -- (3,-2.75) node[anchor=west] {$\partial_2$};

\filldraw[black] (0,0) circle (1.5pt) node[anchor=south east] {$(0,0)$};
\filldraw[black] (1,0) circle (1.5pt) node[anchor=south] {$(1,0)$};
\filldraw[black] (2,0) circle (1.5pt) ;
\filldraw[black] (3,0) circle (1.5pt) ;
\filldraw[black] (4,0) circle (1.5pt) ;
\filldraw[black] (5,0) circle (1.5pt) ;
\filldraw[black] (6,0) circle (1.5pt) ;

\filldraw[black] (0,-1) circle (1.5pt) node[anchor=east] {$(0,-1)$};
\filldraw[black] (1,-1) circle (1.5pt) ;
\filldraw[black] (2,-1) circle (1.5pt) ;
\filldraw[black] (3,-1) circle (1.5pt) ;
\filldraw[black] (4,-1) circle (1.5pt) ;
\filldraw[black] (5,-1) circle (1.5pt) ;
\filldraw[black] (6,-1) circle (1.5pt) ;

\filldraw[black] (0,-2) circle (1.5pt) ;
\filldraw[black] (1,-2) circle (1.5pt) ;
\filldraw[black] (2,-2) circle (1.5pt) ;
\filldraw[black] (3,-2) circle (1.5pt) ;
\filldraw[black] (4,-2) circle (1.5pt) ;
\filldraw[black] (5,-2) circle (1.5pt) ;
\filldraw[black] (6,-2) circle (1.5pt) ;

\filldraw[black] (0,-3) circle (1.5pt) ;
\filldraw[black] (1,-3) circle (1.5pt) ;
\filldraw[black] (2,-3) circle (1.5pt) ;
\filldraw[black] (3,-3) circle (1.5pt) ;
\filldraw[black] (4,-3) circle (1.5pt) ;
\filldraw[black] (5,-3) circle (1.5pt) ;
\filldraw[black] (6,-3) circle (1.5pt) ;

\filldraw[black] (0,-4) circle (1.5pt) ;
\filldraw[black] (1,-4) circle (1.5pt) ;
\filldraw[black] (2,-4) circle (1.5pt) ;
\filldraw[black] (3,-4) circle (1.5pt) ;
\filldraw[black] (4,-4) circle (1.5pt) ;
\filldraw[black] (5,-4) circle (1.5pt) ;
\filldraw[black] (6,-4) circle (1.5pt) ;

\end{tikzpicture}

\end{center}
}

Each $e_{\alpha}$, with $\alpha$ in the fourth quadrant, is a finite number of derivatives away from becoming an unbounded function on $\h$.  This implies
\begin{equation*} 
\bigcap_{k=0}^\infty \cs\left(\h, A^\infty_k\right)= \{(\alpha_1, \alpha_2):\alpha_1\geq 0,\alpha_2 \geq 0\}.
\end{equation*}

 Corollary~\ref{cor-sibony}  thus gives the claimed result.
\end{example}

\begin{remark}
This property of the Hartogs triangle was first proved in Section 5 of \cite{Sibony1975}, by a  different argument.
\end{remark}



\section{Generalized Hartogs Triangles}\label{S:GenHartogs}

Following \cite{EdhMcN16b}, for $\gamma > 0$ define the domains  
\begin{equation}\label{D:genHartogs}
\h_{\gamma} := \left\{(z_1, z_2) \in \C^2: |z_1|^{\gamma} < |z_2| < 1 \right\};
\end{equation}
call $\h_{\gamma}$ the power-generalized Hartogs triangle of exponent $\gamma$. The main result in \cite{EdhMcN16b} is that the Bergman projection $\bm{B}_{\h_\gamma}=\bm{B}_\gamma$ is ``defective" as an $L^p$
operator and, moreover, whether $\gamma\in\Q$ or not determines the extent of its deficiency. The precise result is

\begin{theorem}[\cite{EdhMcN16b}]\label{T:BergmanLp} Let $\h_\gamma$ be given by \eqref{D:genHartogs}.

\begin{itemize}
\item[(i)] Let $\gamma = \frac mn$, where $m,n \in \Z^+$ with $\gcd (m,n) = 1$.   

Then $\bm{B}_\gamma: L^p\left(\mathbb{H}_{\gamma}\right)\to A^p\left(\mathbb{H}_{\gamma}\right)$ boundedly if and only if $p \in \left(\frac{2m+2n}{m+n+1}, \frac{2m+2n}{m+n-1}\right)$.
\smallskip
\item[(ii)] Let $\gamma$ be irrational.

Then $\bm{B}_\gamma: L^p\left(\mathbb{H}_{\gamma}\right)\to A^p\left(\mathbb{H}_{\gamma}\right)$ boundedly if and only if $p = 2$.
\end{itemize}
\end{theorem}

Focus on $\h_{m/n}$, $\frac mn\in\Q$, and integrability exponents $p\geq 2$. 
The proof of (i) in Theorem \ref{T:BergmanLp} actually shows more: the Bergman projection on $\h_{m/n}$ fails to generate $A^p$ functions from $L^p$ data for certain $p$. To apply Theorems \ref{T:surjectivity} and \ref{T:easy_approx}, operators are needed that create $A^p$ functions for $p$ outside the range in Theorem \ref{T:BergmanLp} (i).    

The sub-Bergman projections defined in Section \ref{SS:subBergman} are such operators. Verification of this is done over several sections, leading to 
  
 \begin{theorem}\label{T:SubBergmanLpMapping} Let $\h_{m/n}$,  where $m,n \in \Z^+$ with $\gcd (m,n) = 1$, be given by \eqref{D:genHartogs}.
 
For each $p\ge2$, the sub-Bergman projection $\widetilde{\bm{B}^p}: L^p\left(\h_{m/n}\right) \to A^p\left(\h_{m/n}\right)$ satisfies 

\begin{itemize}
\item[(i)] $\left|\widetilde{\bm{B}^p}\right|$ is bounded on $L^p(\h_{m/n})$
\smallskip
\item[(ii)] $\widetilde{\bm{B}^p}h =h\quad\forall h\in A^p(\h_{m/n})$.
\end{itemize}
\end{theorem}

Theorem \ref{T:SubBergmanLpMapping} contains Theorem \ref{T:intro3} from the Introduction and is proved in Section \ref{SS:ProofSubBergmanLp}.
If $q$ is conjugate to $p$, $\left|\widetilde{\bm{B}^p}\right|$ also maps $L^q(\h_{m/n})$ into $A^q(\h_{m/n})$ boundedly, but the map is no longer surjective, see Remark \ref{R:NonSurjectivitySubBergman}. 
An explicit description of the set of $L^p$-allowable multi-indices plays a crucial role in the proof of Theorem \ref{T:SubBergmanLpMapping}.  
\subsection{Integrability and Orthogonality}

\subsubsection{Holomorphic monomials in $L^p\left(\h_{m/n}\right)$}\label{SS:HolomorphicMonomials}

Let $\h_{m/n}$, $m,n \in \Z^+$ with $\gcd (m,n) = 1$,  be a fixed power-generalized Hartogs triangle throughout the section.   The following calculation was sketched in \cite{EdhMcN16b}.

\begin{lemma}\label{L:(p,m/n)-allowable MultiIndices and Norm}
Let $p \in [1,\infty)$.  The set of $L^p$-allowable multi-indices is
\begin{equation}\label{E:(p,m/n)-allowable MultiIndices}
\cs\left(\h_{m/n},L^p\right) = \left\{ \alpha=(\alpha_1,\alpha_2) : \alpha_1 \ge 0, \ \ n\alpha_1+m\alpha_2\ge \left\lfloor -\frac{2}{p}(m+n) +1 \right\rfloor \right\}.
\end{equation}
For $\alpha \in \cs(\h_{m/n},L^p)$, 
\begin{equation}\label{E:(p,m/n)-allowable Norm}
\norm{e_{\alpha}}_{L^p(\h_{m/n})}^p = \frac{4m\pi^2}{n(p\alpha_1+2)^2 + m(p\alpha_1+2)(p\alpha_2+2)}
\end{equation}
\end{lemma}
\begin{proof}
Note there are points in $\h_{m/n}$ where $z_1=0$, which forces $\alpha_1 \ge 0$.  Computing in polar coordinates 
\begin{align*}
\int_{\h_{m/n}} \left|z^{\alpha} \right|^p \,dV 
&= 4\pi^2 \int_0^1 r_2^{p\alpha_2+1} \int_0^{r_2^{n/m}} r_1^{p\alpha_1+1} \,dr_1 dr_2 \\
&= \frac{4\pi^2}{p\alpha_1+2}\int_0^1 r_2^{p\alpha_2+1+\frac nm(p\alpha_1+2)} \, dr_2.
\end{align*}
This integral converges if and only if the exponent $p\alpha_2+1+\frac nm(p\alpha_1+2) > -1$.  From here, \eqref{E:(p,m/n)-allowable Norm} easily follows.  To see \eqref{E:(p,m/n)-allowable MultiIndices}, notice that since $\alpha_1, \alpha_2, m,n\in \Z$,
\begin{equation}\label{E:(p,m/n)-allowable MultiIndices NONstrict}
p\alpha_2+2+\frac nm (p\alpha_1+2)>0 \ \Longleftrightarrow \ n\alpha_1+m\alpha_2 \ge \left\lfloor -\frac{2}{p}(m+n) +1 \right\rfloor.
\end{equation}

\end{proof}

Examine the sets $\cs\left(\h_{m/n},L^p\right)$ as functions of $p\in [1,\infty)$. The floor function in \eqref{E:(p,m/n)-allowable MultiIndices} shows that
$$\cs\left(\h_{m/n},L^p\right)=\cs\left(\h_{m/n},L^{p\pm\epsilon}\right)$$
if $\epsilon >0$ is small, unless $-\frac{2}{p}(m+n) +1\in\Z$. The lattice points in $\cs\left(\h_{m/n},L^p\right)$ are therefore stable except for certain exceptional $p$. Call these exceptional values {\it thresholds}. Note that $\cs\left(\h_{m/n},L^t\right)\subset \cs\left(\h_{m/n},L^s\right)$ if $s<t$, so $\cs\left(\h_{m/n},L^p\right)$ jumps to a smaller set of lattice points as $p$ increases past a threshold value.

The next result makes this stabilization precise and shows there are only a finite number of thresholds for a given $\h_{m/n}$.

\begin{proposition}\label{P:ApClasses}
There are exactly $2m+2n$ thresholds associated to $\h_{m/n}$.  They occur when $p_k=\frac{2m+2n}{1-k}$ for $k\in\{ 1-2m-2n, 2-2m-2n,\dots, -1, 0\}$. 

Consider the corresponding partition of $[1,\infty)$ 
\begin{equation}\label{E:partition}
\left[1,\infty \right) = \bigcup_{k= 1-2m-2n}^0 \left[\,p_k,p_{k+1}\right),\qquad p_k=\frac{2m+2n}{1-k}.
\end{equation}
 Then for any $p\in[\,p_k,p_{k+1})$, 
\begin{equation}\label{E:pkAllowableMultiIndices}
\cs\left(\h_{m/n},L^p\right) = \left\{(\alpha_1,\alpha_2) : \alpha_1 \ge 0, \ \ n\alpha_1+m\alpha_2 \ge k \right\}=\cs\left(\h_{m/n},L^{\,p_k}\right),
\end{equation}
and
\begin{equation}\label{E:boundedAllowableMultiIndices}
\cs\left(\h_{m/n},L^{\infty}\right) = \left\{(\alpha_1,\alpha_2) : \alpha_1 \ge 0, \ \ n\alpha_1+m\alpha_2 \ge 0 \right\} = \cs\left(\h_{m/n},L^{2m+2n}\right),
\end{equation}
\end{proposition}

\medskip
\begin{remark}
\eqref{E:boundedAllowableMultiIndices} says every $e_\alpha\in A^{2m+2n}\left(\h_{m/n}\right)$ is necessarily bounded. This generalizes statement \eqref{E:H1BoundedMonomials} on $\h_1$.
\end{remark}

\begin{proof}
Define 
$\ell_{m,n}(p) := -\frac{2}{p}(m+n) +1,\, p\in [1,\infty)$.
The function $\ell_{m,n}(p)$ is increasing and takes values in the interval $[1-2m-2n, 1)$.  Note $\ell_{m,n}(p) = k \in \Z$ if and only if
$p = \frac{2m+2n}{1-k}$.

Rewrite the partition in \eqref{E:partition}:
\begin{align*}\label{E:IkIntervalDecomp}
\left[1,\infty\right) &= \bigcup_k \left[\frac{2m+2n}{1-k},\frac{2m+2n}{-k} \right) := \bigcup_k J_k,
\end{align*}
where the union is taken over $k\in\{ 1-2m-2n, 2-2m-2n,\dots, -1, 0\}$.  
Suppose $p,p'\in J_k$  for some $J_k$.  Then 
\begin{align*}
k\in \left(-\frac{2}{p}(m+n),-\frac{2}{p}(m+n)+1 \right] \cap \left(-\frac{2}{p'}(m+n),-\frac{2}{p'}(m+n)+1 \right],
\end{align*}
which in turn implies
$\left\lfloor \ell_{m,n}(p) \right\rfloor = k =  \left\lfloor \ell_{m,n}(p') \right\rfloor$, and
shows \eqref{E:pkAllowableMultiIndices} holds.

To see  \eqref{E:boundedAllowableMultiIndices}, let $\alpha = (\alpha_1,\alpha_2) \in \cs\left(\h_{m/n},L^{2m+2n}\right)$. Equation \eqref{E:(p,m/n)-allowable MultiIndices} says that $\alpha_1 \ge 0$ and $n\alpha_1+m\alpha_2 \ge 0$.  Since $|z_1|^m < |z_2|^n<1$ if $z\in\h_{m/n}$, it follows that
\begin{align*}
\left|z_1^{\alpha_1}z_2^{\alpha_2} \right|^m = \left|\frac{z_1^m}{z_2^n} \right|^{\alpha_1} \left|z_2\right|^{n\alpha_1+m\alpha_2} <1,
\end{align*}
which says $\alpha\in \cs\left(\h_{m/n},L^{\infty}\right)$.
\end{proof}








\subsubsection{An example; pairing Monomials}\label{SS:Pairing Monomials}

Consider the domain $\h_2$. Proposition \ref{P:ApClasses} says there are $6$ thresholds associated to $\h_2$:
\bigskip

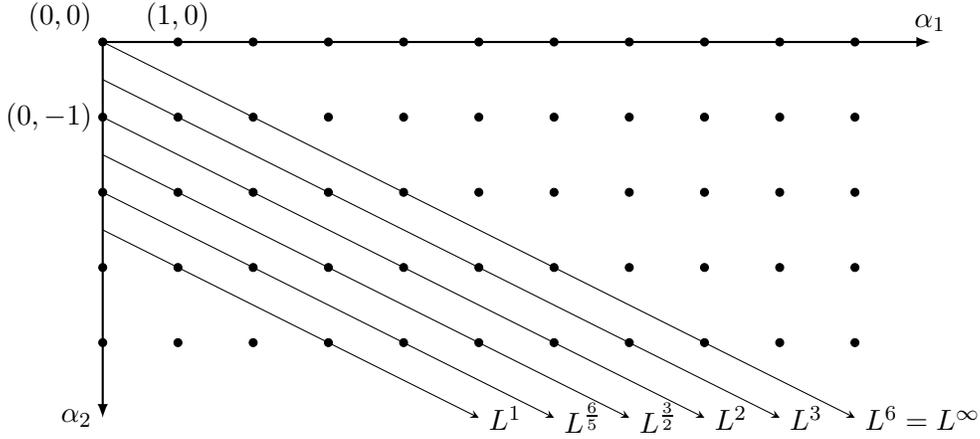
{
\centering

\begin{tikzpicture}

\draw[-{latex}, thick] (0,0) -- (11,0) node[anchor=south] {$\alpha_1$};
\draw[-{latex}, thick] (0,0) -- (0,-5) node[anchor=east] {$\alpha_2$};

\draw[-stealth, thin] (0,0) -- (10,-5) node[anchor=west] {$L^6 =L^\infty$}; 
\draw[-stealth, thin] (0,-.5) -- (9,-5)  node[anchor=west] {$L^3$}; 
\draw[-stealth, thin] (0,-1) -- (8,-5)  node[anchor=west] {$L^2$}; 
\draw[-stealth, thin] (0,-1.5) -- (7,-5) node[anchor=west] {$L^\frac32$}; 
\draw[-stealth, thin] (0,-2) -- (6,-5)  node[anchor=west] {$L^\frac65$}; 
\draw[-stealth, thin] (0,-2.5) -- (5,-5)  node[anchor=west] {$L^1$}; 

\filldraw[black] (0,0) circle (1.5pt) node[anchor=south east] {$(0,0)$};
\filldraw[black] (1,0) circle (1.5pt) node[anchor=south] {$(1,0)$};
\filldraw[black] (2,0) circle (1.5pt) ;
\filldraw[black] (3,0) circle (1.5pt) ;
\filldraw[black] (4,0) circle (1.5pt) ;
\filldraw[black] (5,0) circle (1.5pt) ;
\filldraw[black] (6,0) circle (1.5pt) ;
\filldraw[black] (7,0) circle (1.5pt) ;
\filldraw[black] (8,0) circle (1.5pt) ;
\filldraw[black] (9,0) circle (1.5pt) ;
\filldraw[black] (10,0) circle (1.5pt) ;

\filldraw[black] (0,-1) circle (1.5pt) node[anchor=east] {$(0,-1)$};
\filldraw[black] (1,-1) circle (1.5pt) ;
\filldraw[black] (2,-1) circle (1.5pt) ;
\filldraw[black] (3,-1) circle (1.5pt) ;
\filldraw[black] (4,-1) circle (1.5pt) ;
\filldraw[black] (5,-1) circle (1.5pt) ;
\filldraw[black] (6,-1) circle (1.5pt) ;
\filldraw[black] (7,-1) circle (1.5pt) ;
\filldraw[black] (8,-1) circle (1.5pt) ;
\filldraw[black] (9,-1) circle (1.5pt) ;
\filldraw[black] (10,-1) circle (1.5pt) ;

\filldraw[black] (0,-2) circle (1.5pt) ;
\filldraw[black] (1,-2) circle (1.5pt) ;
\filldraw[black] (2,-2) circle (1.5pt) ;
\filldraw[black] (3,-2) circle (1.5pt) ;
\filldraw[black] (4,-2) circle (1.5pt) ;
\filldraw[black] (5,-2) circle (1.5pt) ;
\filldraw[black] (6,-2) circle (1.5pt) ;
\filldraw[black] (7,-2) circle (1.5pt) ;
\filldraw[black] (8,-2) circle (1.5pt) ;
\filldraw[black] (9,-2) circle (1.5pt) ;
\filldraw[black] (10,-2) circle (1.5pt) ;

\filldraw[black] (0,-3) circle (1.5pt) ;
\filldraw[black] (1,-3) circle (1.5pt) ;
\filldraw[black] (2,-3) circle (1.5pt) ;
\filldraw[black] (3,-3) circle (1.5pt) ;
\filldraw[black] (4,-3) circle (1.5pt) ;
\filldraw[black] (5,-3) circle (1.5pt) ;
\filldraw[black] (6,-3) circle (1.5pt) ;
\filldraw[black] (7,-3) circle (1.5pt) ;
\filldraw[black] (8,-3) circle (1.5pt) ;
\filldraw[black] (9,-3) circle (1.5pt) ;
\filldraw[black] (10,-3) circle (1.5pt) ;

\filldraw[black] (0,-4) circle (1.5pt) ;
\filldraw[black] (1,-4) circle (1.5pt) ;
\filldraw[black] (2,-4) circle (1.5pt) ;
\filldraw[black] (3,-4) circle (1.5pt) ;
\filldraw[black] (4,-4) circle (1.5pt) ;
\filldraw[black] (5,-4) circle (1.5pt) ;
\filldraw[black] (6,-4) circle (1.5pt) ;
\filldraw[black] (7,-4) circle (1.5pt) ;
\filldraw[black] (8,-4) circle (1.5pt) ;
\filldraw[black] (9,-4) circle (1.5pt) ;
\filldraw[black] (10,-4) circle (1.5pt) ;

\end{tikzpicture}

}
\bigskip\bigskip

\noindent The lines come from  \eqref{E:pkAllowableMultiIndices}.  The lattice points on the first five lines represent $L^p$-integrable monomials for all $p$ up to {\em but not including} the $p$ value of the next line, while the lattice points {\em on and above} the $p = 6$ line correspond to bounded monomials on $\h_2$.  

Choose $\beta \in \cs\left(\h_{2},L^{\frac 32}\right)$ and $\delta \in \cs\left(\h_{2},L^{2}\right)$ with $\beta \ne \delta$. The first observation is that the $L^2$ pairing 
\begin{equation}\label{E:littlePairing}
\left\langle e_\beta, e_\delta\right\rangle_{\h_2}
\end{equation}
is defined. Note that $\frac 32$ and $2$ are not conjugate. If $\beta$ also belonged to $\cs\left(\h_{2},L^{2}\right)$, H\" older's inequality would imply \eqref{E:littlePairing} is finite. Thus assume $\beta$ lies on the line $L^{\frac 32}$ in the diagram. Proposition \ref{P:ApClasses} says $e_\beta\in L^t(\h_2)$ for all $t<2$ and that $e_\delta\in L^s(\h_2)$ for all $s<3$. There are infinitely many pairs of conjugate exponents in these two intervals, so once again H\" older's inequality shows 
\eqref{E:littlePairing} is defined.
The second observation is that \eqref{E:littlePairing} equals 0. This follows since $\beta \ne \delta$ and the monomials $\{e_\alpha\}$ are orthogonal on $\h_2$.

The same conclusion holds for any multi-indices $\beta \ne \delta$ chosen with $\beta \in \cs(\h_{2},L^{\frac 65})$ (respectively $\beta \in \cs(\h_{2},L^{1})$) and $\delta \in \cs(\h_{2},L^{3})$ (respectively $\delta \in \cs(\h_{2},L^{6})$).  The following corollary of Proposition \ref{P:ApClasses} gives the general version:

\begin{corollary}\label{C:ExtendedOrthogonalityPairing}
Let $\gamma = \frac mn$, $k\in\{ 1-2m-2n, 2-2m-2n,\dots, -1, 0\}$, and define $j(k) := 1-k-2m-2n$. Set
\begin{equation*}
p_k = \frac{2m+2n}{1-k}, \qquad p_{j(k)} = \frac{2m+2n}{1-j(k)} = \frac{2m+2n}{2m+2n+k}.
\end{equation*}
Then for any choice of multi-indices $\beta \in \cs(\h_{\gamma},L^{\,p_k})$ and $\delta \in \cs(\h_{\gamma},L^{\,p_{j(k)}})$ with $\beta \ne \delta$, the inner product
\begin{equation*}
\left\langle e_\beta, e_\delta\right\rangle_{\h_{\gamma}} = 0.
\end{equation*}
\end{corollary}
\begin{remark} Corollary \ref{C:ExtendedOrthogonalityPairing} is nontrivial only because $p_k$ and $p_{j(k)}$ are not conjugate; indeed, $\frac{1}{p_k} + \frac{1}{p_{j(k)}} > 1$.  No analogue of Corollary \ref{C:ExtendedOrthogonalityPairing} exists for $\h_\gamma$, $\gamma\notin\Q$.
\end{remark}


\subsection{Constructing $A^p$ functions}\label{S:CreatingAp}

Construction of the sub-Bergman kernels and projection operators is based on the decomposition of monomials in Proposition \ref{P:ApClasses}.  

\subsubsection{Type-$A$ operators on $\h_{m/n}$}\label{SS:TypeA}

A lemma from \cite{EdhMcN16b} is recalled that relates estimates on a class of kernels defined on $\h_{m/n} \times \h_{m/n}$ to mapping properties of the associated integral operators.  If $\Omega \subset \C^n$ is a domain and $K$ is an a.e. positive, measurable function on $\Omega \times \Omega$, let $\ck$ denote the integral operator associated to $K$:
\begin{equation*}\label{E:op_kernel}
\ck f (z) = \int_{\Omega} K(z,w) f(w) \,dV(w).
\end{equation*}

\begin{definition} \label{D:typeA} For $A\in\R^+$, call $\ck$ an operator of type-$A$ on $\h_{m/n}$ if its kernel satisfies
\begin{equation*}\label{E:typeA}
 K\left(z_1,z_2, w_1,w_2\right) \lesssim \frac{\left| z_2w_2\right|^A}{\left|1-z_2\bar w_2\right|^2\, \left| z_2^n\bar w_2^n-z_1^m\bar w_1^m\right|^2},
\end{equation*}
for a constant independent of $(z,w)\in \h_{m/n}\times\h_{m/n}$.
\end{definition}

The basic $L^p$ mapping result is 

\begin{proposition}[\cite{EdhMcN16b}]\label{P:typeA_mapping}
If $\ck$ is an operator of type-$A$ on $\h_{m/n}$, then $\ck:L^p\left(\h_{m/n}\right)\longrightarrow L^p\left(\h_{m/n}\right)$ boundedly if
\begin{equation}\label{E:typeA_range}
\frac{2n+2m}{Am+2n+2m-2nm} < p <\frac{2n+2m}{2nm -Am},
\end{equation}
whenever \begin{equation}\label{E:boundsA}
n(2-m^{-1})-1 < A < 2n.
\end{equation}
\end{proposition}

\begin{remark}
The range of $L^p$ boundedness as $A$ tends to the upper and lower bounds in \eqref{E:boundsA} is significant.  As $A\to 2n$, the interval in \eqref{E:typeA_range} increases to $(1,\infty)$; thus an operator of type-$2n$ on $\h_{m/n}$ is $L^p$ bounded for all $1<p<\infty$.  In the other direction, note the left endpoint $n(2-m^{-1})-1 \ge 0$ for all choices of $m,n \in \Z^+$.  As $A$ decreases to this endpoint, the interval in \eqref{E:typeA_range} collapses towards the point $\{2\}$. However an operator of type $n(2-m^{-1})-1$ is not necessarily bounded on any $L^p$ space, including $L^2$.
\end{remark}

\subsubsection{Splitting monomials by integrability class}

Abbreviate the $L^p$-allowable multi-indices given by Proposition \ref{P:ApClasses}: 
\begin{align*} \label{D:MonomialSetS_k}
\cs(\h_{m/n},L^{p_k}) &= \left\{(\alpha_1,\alpha_2) : \alpha_1 \ge 0, \ \ n\alpha_1+m\alpha_2 \ge k \right\} \notag := S_{k},
\end{align*}
where $p_k = \frac{2m+2n}{1-k}$ and $k\in\{ 1-2m-2n, 2-2m-2n,\dots, -1, 0\}$.  

The $L^p$ sub-Bergman kernels for $p\ge 2$ are defined
\begin{equation}\label{D:DefOfBp}
\widetilde{B^{p}}(z,w) := \sum_{\alpha\in S_k} \frac{e_{\alpha}(z)\ol{e_{\alpha}(w)}}{\left\|e_{\alpha}\right\|_2^{2}},\qquad p\in[\,p_k,p_{k+1}).
\end{equation}
The stabilization in Proposition \ref{P:ApClasses} accounts for the identical definition of $\widetilde{B^{p}}(z,w)$ for all $p\in[\,p_k,p_{k+1})$. Note that only $S_k$ for
$k\in [1-m-n, 0]$ occurs in any of the kernels \eqref{D:DefOfBp}, since $p\ge 2$. 
Proposition \ref{P:ApClasses} also says $S_0 = \cs\left(\h_{m/n},L^{2m+2n}\right) = \cs\left(\h_{m/n},L^{\infty}\right)$. Consequently,  denote the sum
\begin{align}\label{D:DefOfB^{infty}}
\sum_{\alpha\in S_0} \frac{e_{\alpha}(z)\ol{e_{\alpha}(w)}}{\left\|e_{\alpha}\right\|_2^{2}} := \widetilde{B^{\infty}}(z,w)
\end{align}
and call $\widetilde{B^{\infty}}(z,w)$ the $L^{\infty}$ sub-Bergman kernel on $\h_{m/n}$. The sum defining $\widetilde{B^{\infty}}(z,w)$ consists only of $L^{\infty}$ monomials.

As an aid to calculating the sums \eqref{D:DefOfBp} and \eqref{D:DefOfB^{infty}}, define
\begin{equation}\label{E:SkDiagonalMultInd}
s_k = \left\{\alpha : \alpha_1 \ge 0,\, n\alpha_1 + m\alpha_2 = k\right\},
\end{equation}  
and consider the functions
\begin{equation}\label{E:SubkernelComp1}
b^{\,p_k}(z,w) = \sum_{\alpha \in s_k} \frac{e_{\alpha}(z)\ol{e_{\alpha}(w)}}{\left\|e_{\alpha}\right\|_2^{2}}.
\end{equation}
Orthogonality of $\{e_\alpha\}$ yields the decomposition
\begin{equation}\label{D:LpkSubBergmanKernel}
\sum_{j=k}^{-1} b^{\,p_j}(z,w) + \widetilde{B^{\infty}}(z,w) = \widetilde{B^{p_k}}(z,w)
\end{equation}
for negative integers $k\ge1-m-n$.

\subsubsection{Analyzing the sub-Bergman kernels.}\label{SS:AnalyzeSubBergman}
  
The first step is to obtain an upper bound on $b^{\,p_k}$ connected to Definition \ref{D:typeA}.
\begin{proposition}
The following estimate holds for all $z,w \in \h_{m/n}$
\begin{equation}\label{E:subkernelBound}
\left|b^{\,p_k}(z,w)\right| \lesssim \frac{|z_2\bar{w}_2|^{2n+\frac km}}{|z_2^n\bar{w}_2^n-z_1^m\bar{w}_1^m|^2}.
\end{equation}
Recall $k<0$ in \eqref{E:subkernelBound}, $k\in\{ 1-m-n, 2-m-n,\dots, -1\}$.
\end{proposition}
\begin{proof}
Since $\gcd (m,n) = 1$, there is a unique pair $(\beta_1,\beta_2)$ with $0 \le \beta_1 \le m-1$ and $n\beta_1 + m\beta_2 = k$.  Notice that the subsequent lattice points on this line are of the form $(\beta_1+jm, \beta_2-jn)$. Equation \eqref{E:(p,m/n)-allowable Norm} says for all $\alpha \in \cs\left(\h_{m/n},L^2\right)$,
\begin{align}\label{E:L^2coefficient}
\left\|e_{\alpha}\right\|_2^{2}  &= \frac{m\pi^2}{(\alpha_1 + 1)(n\alpha_1+m\alpha_{2}+m+n)}.
\end{align}

In what follows, let $s:=z_1\bar{w}_1$ and $t:=z_2\bar{w_2}$. Definition \eqref{E:SubkernelComp1} and \eqref{E:L^2coefficient} imply
\begin{align}
b^{\,p_k}(z,w) &= \frac{m+n+k}{m\pi^2} \sum_{j=0}^{\infty} \left(\beta_1+jm+1 \right)s^{\beta_1+jm}t^{\beta_2 - jn} \notag \\
&= \frac{m+n+k}{m\pi^2} \cdot t^{k/m} \sum_{j=0}^{\infty} \left(\beta_1+jm+1 \right)s^{\beta_1+jm}(t^{-n/m})^{\beta_1 + jm} \notag \\
&= \frac{m+n+k}{m\pi^2} \cdot t^{k/m} \sum_{j=0}^{\infty} \left(\beta_1+jm+1 \right)u^{\beta_1+jm} \label{E:SubkernelComp2}
\end{align}
where $u := st^{-n/m}$.  Writing this series in closed form yields 
\begin{align*}
\eqref{E:SubkernelComp2} &= \frac{m+n+k}{m\pi^2} \cdot t^{k/m} u^{\beta_1} \cdot \frac{(\beta_1+1)+(m-\beta_1-1)u^m}{(1-u^m)^2} \\
&= \frac{m+n+k}{m\pi^2} \cdot s^{\beta_1}t^{\beta_2} \cdot \frac{(\beta_1+1)t^{2n}+(m-\beta_1-1)s^mt^n}{(t^n-s^m)^2}. 
\end{align*}
Noting that $|s|^m < |t|^n$, the bound \eqref{E:subkernelBound} follows.
\end{proof}

Let $\bm{b}^{\,p_k}$ be the integral operator
\begin{equation}
\bm{b}^{\,p_k}(f)(z) := \int_{\h_{m/n}}b^{\,p_k}(z,w)f(w)\,dV(w)
\end{equation}
The operator $\bm{b}^{\,p_k}$ is orthogonal projection from $L^2(\h_{m/n}) \to \ol{\rm{span}}_{L^2}\left\{e_{\alpha}: \alpha\in s_k\right\}$.  Note each $s_k$ is a set of points in the lattice point diagram lying on a single line.

\begin{corollary}
Let $p_k =  \frac{2m+2n}{1-k} $ for each integer $1-m-n \le k \le -1$ and $q_k$ be conjugate to $p_k$.  The projection $\bm{b}^{\,p_k}$ is an operator of type-{A} for $A = 2n +\frac{k}{m}$.  Thus, $\bm{b}^{\,p_k}$ is $L^p$-bounded for 
\begin{equation}\label{E:bkLpBoundedness}
p \in \left(\frac{2n+2m}{2n+2m+k}, \frac{2n+2m}{-k} \right) = (q_{k+1},p_{k+1}).
\end{equation}
\end{corollary}
\begin{proof}
Set $A = 2n +\frac{k}{m}$ in Proposition \ref{P:typeA_mapping}.
\end{proof}


The second step is to show the kernel $\widetilde{B^{\infty}}(z,w)$ satisfies bounds related to Definition \ref{D:typeA} and is more involved.

\begin{proposition}\label{T:BInftyEstimate}
The $L^{\infty}$ sub-Bergman kernel on $\h_{m/n}$ satisfies
\begin{equation}\label{E:BinftyEstimate}
\left| \widetilde{B^{\infty}}(z,w)\right|\lesssim \frac{\left| z_2\bar{w_2}\right|^{2n}}{\left|1-z_2\bar w_2\right|^2\, \left| z_2^n\bar w_2^n-z_1^m\bar w_1^m\right|^2}.
\end{equation}
\end{proposition}

\begin{proof} Recall the description of $\cs(\h_{m/n},L^{\infty})$ given by \eqref{E:boundedAllowableMultiIndices} and let $r\in \{0,1,\dots,m-1 \}$.  Since $\gcd (m,n)=1$, there is a unique $(\alpha_1,\alpha_2)$  with both $n\alpha_1+m\alpha_2=r$ and $0\le \alpha_1 \le m-1$.  Set this $\alpha_1$ = $\sigma(r)$.  The function $\sigma$ is a permutation of the set $\{0,1,\dots,m-1 \}$ with $\sigma(0)=0$.

Each $\alpha \in \cs(\h_{m/n},L^{\infty}) = \{(\alpha_1,\alpha_2) : \alpha_1\ge0,\ n\alpha_1+m\alpha_2 \ge0 \}$ can uniquely described by a line of the form $n\alpha_1+m\alpha_2 = k$ and an $\alpha_1$ value.  Again letting $r\in \{0,1,\dots,m-1 \}$, parametrize $k$ and $\alpha_1$ by
\begin{align*}
n\alpha_1+m\alpha_2 = md+r, \qquad &d=0,1,\dots \\
\alpha_1 = mj+\sigma(r), \qquad &j=0,1,\dots
\end{align*}

For ease of notation set $s = z_1\bar{w}_1, t = z_2\bar{w}_2$.  From equations \eqref{D:DefOfBp} and \eqref{E:L^2coefficient},
\begin{align}
\wt{B^{\infty}}(z,w) &= \frac{1}{m \pi^2}\sum_{\alpha\in\cs({\h_{m/n},L^{\infty}})} (\alpha_1+1)(n\alpha_1+m\alpha_2+m+n)s^{\alpha_1}t^{\alpha_2} \notag \\
&= \frac{1}{m \pi^2} \sum_{r=0}^{m-1} \sum_{d,j=0}^{\infty} (mj+\sigma(r)+1)(md+r+m+n) s^{mj+\sigma(r)}t^{d+\frac{r}{m}-nj-\frac{n}{m}\sigma(r)} \notag \\
&= \frac{1}{m\pi^2}\sum_{r=0}^{m-1} u^{\sigma(r)} t^{\frac rm} \Bigg( \sum_{j=0}^{\infty}(mj+ \sigma(r)+1)u^{mj} \Bigg) \Bigg( \sum_{d=0}^{\infty}(md+r+m+n)t^d \Bigg)  \notag \\
&:= \frac{1}{m\pi^2}\sum_{r=0}^{m-1} u^{\sigma(r)} t^{\frac rm}  I_r(u) J_r(t) \label{D:I_r(u)ANDJ_r(t)},
\end{align}
where we have introduced the new variable $u = st^{-n/m}$.  Note both $|t|<1$ and $|u|<1$ on $\h_{m/n}$.  For fixed $r$, estimate the sums $I_r(u)$ and $J_r(t)$ given in \eqref{D:I_r(u)ANDJ_r(t)}:
\begin{align}
|I_r(u)| = \left| \sum_{j=0}^{\infty} (mj+1)u^{mj} + \sigma(r) \sum_{j=0}^{\infty} u^{mj} \right| \lesssim \frac{1}{|1-u^m|^2} = \frac{|t|^{2n}}{|t^n-s^m|^2}, \label{E:I_r(u)Bound}
\end{align}
and
\begin{align}
|J_r(t)| = \left| m\sum_{d=0}^{\infty} (d+1)t^d + (r+n)\sum_{d=0}^{\infty} t^d \right| \lesssim \frac{1}{|1-t|^2}. \label{E:J_r(t)Bound}
\end{align}
Note both bounds hold for all $r\in\{0,1,\dots,m-1 \}$.  Combining \eqref{E:I_r(u)Bound} and \eqref{E:J_r(t)Bound} with \eqref{D:I_r(u)ANDJ_r(t)} gives the result.

\end{proof}

\subsection{Proof of Theorem \ref{T:SubBergmanLpMapping}}\label{SS:ProofSubBergmanLp}

For $p\in [2,\infty)$, the $L^{\,p}$ sub-Bergman projection  is
\begin{equation*}
\widetilde{\bm{B}^{p}}f(z) := \int_{\h_{m/n}}\widetilde{B^{p}}(z,w)f(w)\,dV(w),
\end{equation*}
with kernel given by  \eqref{D:DefOfBp}. 
Notice the identical kernels in definition \eqref{D:DefOfBp} imply $\widetilde{\bm{B}^{p}}=\widetilde{\bm{B}^{p'}}$ for all $p, p'\in [p_k, p_{k+1})$.
 Similarly, $\widetilde{\bm{B}^{\infty}}$ denotes the $L^{\infty}$ sub-Bergman projection on $\h_{m/n}$,  the operator whose kernel is defined by \eqref{D:DefOfB^{infty}}.

\begin{proposition}\label{T:TechnicalSubBergmanStatement}
Let $p_k = \frac{2m+2n}{1-k}, $ for $k\in\{ 1-m-n, 2-m-n,\dots, -1\}$, and let $q_k$ denote the conjugate exponent of $p_k$.  Interpret $p_1=\infty$ and $q_1=1$.  

Let $p\in [p_k,p_{k+1})$.  The following hold:

\begin{itemize}
\item[(i)] $\left|\widetilde{\bm{B}^{p}}\right|$ is  $L^{p'}$ bounded for all $p'\in (q_{k+1},p_{k+1})$. 
\medskip
\item[(ii)]  $\left|\widetilde{\bm{B}^{\infty}}\right|$ is a bounded operator on $L^p$ for all $p\in (1,\infty)$.
\end{itemize}
\end{proposition}
\begin{proof}
  Estimate \eqref{E:BinftyEstimate} shows that $\left|\widetilde{\bm{B}^{\infty}}\right|$ is a type-$A$ operator with $A=2n$. Proposition \ref{P:typeA_mapping} then implies (ii).  For $1-m-n \le k \le -1$, apply the triangle inequality to equation \eqref{D:LpkSubBergmanKernel} together with estimate \eqref{E:subkernelBound} to see that $\left|\widetilde{\bm{B}^{p_k}}\right|$ is a type-$A$ operator with $A = 2n +\frac{k}{m}$.  Proposition \ref{P:typeA_mapping} then implies (i).\end{proof}

To complete the proof of Theorem \ref{T:SubBergmanLpMapping}, recall that $\widetilde{\bm{B}^p}$ is defined as the orthogonal projection from $L^2(\h_{m/n})$ onto $G^{2,p}(\h_{m/n})$, the target space given by equation \eqref{eq-m2pdef}.  Since $A^p(\h_{m/n}) \subset G^{2,p}(\h_{m/n})$, reproduction property (ii) of Theorem \ref{T:surjectivity} holds. \qed

\begin{remark}\label{R:NonSurjectivitySubBergman}
Again let $p\ge 2$ with $p \in [p_k,p_{k+1})$.  If $p'\in (q_{k+1},p_{k+1})$, then its conjugate $q'\in (q_{k+1},p_{k+1})$.  Proposition \ref{T:TechnicalSubBergmanStatement} shows $\left|\widetilde{\bm{B}^{p}}\right|$ is both $L^{p'}$ and $L^{q'}$ bounded.  In particular, $\left|\widetilde{\bm{B}^{p}}\right|$ is bounded on $L^q(\h_{m/n})$, where $q$ is conjugate to $p$.

On the other hand, reproduction of the space $A^{q'}$ fails for all $q'<2$.  Indeed, a slight modification of the proof of Proposition \ref{P:BonA^p} shows:  if $f \in A^{q'}(\h_{m/n})$, then
\begin{equation*}
\widetilde{\bm{B}^{p}}(f)(z) = \sum_{\alpha \in \cs(\h_{m/n},L^p)} a_{\alpha}(f)e_{\alpha}(z).
\end{equation*}
Lemma \ref{L:(p,m/n)-allowable MultiIndices and Norm} implies $\cs(\h_{m/n},L^{q'})$ is a {\em strict} superset of  $\cs(\h_{m/n},L^{2})$ which in turn contains $\cs(\h_{m/n},L^{p})$.  Thus non-trivial elements in $A^{q'}$ are mapped to 0.  Ramifications of this are seen in the next subsection.
\end{remark}


\subsection{Duality, Approximation and Minimization}\label{S:dual}

The sub-Bergman projections give precise answers to versions of (Q1-3) on the domains $\h_{m/n}$.  

\subsubsection{Duality}  The dual space of $A^p\left(\h_{m/n}\right)$ for all $1<p<\infty$ can be concretely described. The representation is particularly cogent when $p > 2$.

\begin{proposition}\label{T:DualityGenHartogs}
Let $p>2$ with conjugate $q$.  The dual space $A^p\left(\h_{m/n}\right)'$ can be identified with a proper subset of $A^q(\h_{m/n})$.  Namely,
\begin{equation}\label{E:ApDualH}
A^p(\h_{m/n})' \cong   \left\{ f\in A^q(\h_{m/n}) : f = \sum_{\alpha \in  \cs\left(\h_{m/n},L^p\right)} a_{\alpha}(f)e_{\alpha} \right\}.
\end{equation}
Additionally,
\begin{equation}\label{E:AqDualH}
A^q(\h_{m/n})' \cong  A^p(\h_{m/n}) \oplus \overline{\rm{span}}_{A^q(\h_{m/n})'}\left\{a_\alpha: \alpha\in \cs\left(\h_{m/n},L^q\right) \backslash \cs\left(\h_{m/n},L^p\right) \right\}.
\end{equation}
\end{proposition}

\begin{proof}
Since $\left|\widetilde{\bm{B}^{p}}\right|$ is bounded on $L^p$, Proposition \ref{P:ReinhardtDuals} applies.  Equation \eqref{E:ApDualH} follows from part (ii) of Proposition  \ref{P:ReinhardtDuals}, noting the right hand side of \eqref{E:ApDualH} is $G^{q,p}(\h_{m/n})$.  Equation \eqref{E:AqDualH} follows from part (iii) of the same proposition.
\end{proof}

This result  should be compared with the breakdown shown in Section \ref{SS:FailureRep}.

\subsubsection{Approximation of $A^p$ functions} The form of (Q2) addressed is the following: given $p\in (1,\infty)$ and $r >p$, when can $f\in A^p\left(\h_{m/n}\right)$ be approximated by $A^r\left(\h_{m/n}\right)$ functions in the $L^p$ norm?  As in Proposition \ref{T:DualityGenHartogs}, the answer is most appealing when $p>2$.

\begin{proposition}\label{T:ApproxGenHartogs1} Let $p\geq 2$ be given and $r>p$. Then $f\in A^p(\h_{m/n})$ can be approximated by $A^r\left(\h_{m/n}\right)$ functions in the $L^p$ norm if and only if
$\widetilde{\bm{B}^r}f = f$.
\end{proposition}

\begin{proof}
Suppose $f\in A^p\left(\h_{m/n}\right)$ and $\widetilde{\bm{B}^r}f = f$. 
Proceed as in the proof of Proposition \ref{T:easy_approx}.
 Since $f \in L^p\left(\h_{m/n}\right)$, there is a sequence $\phi_n\in C_c^{\infty}\left(\h_{m/n}\right)$ satisfying $\norm{\phi_n - f}_p \to 0$ as $n\to \infty$.  Set $f_n := \widetilde{\bm{B}^r}\phi_n$.  Note $f_n\in A^r\left(\h_{m/n}\right)$ by Proposition \ref{T:TechnicalSubBergmanStatement} . Moreover
 \begin{align*}
\norm{f_n-f}_p = \left\|{\widetilde{\bm{B}^r}(\phi_n - f)}\right\|_p \lesssim \norm{\phi_n - f}_p,
\end{align*}
so $f$ is approximable as claimed.

For the converse, suppose $f\in A^p(\h_{m/n})$ and $\widetilde{\bm{B}^r}f \neq f$. 
By Proposition \ref{prop-subbergman2} , there exists $e_\beta\in\cs\left(\h_{m/n}, L^p\right)\setminus \cs\left(\h_{m/n}, L^r\right)$ such that $a_\beta(f)\neq 0$, with $a_\beta(f)$ associated to $f$ via \eqref{E:ApLaurentSeries}.

Suppose there were a sequence $g_n\in A^r(\h_{m/n})$ such that $g_n\to f$ in $A^p\left(\h_{m/n}\right)$. Note that $a_\beta(g_n)=0$ for all $n$. Thus $a_\beta(g_n-f) =-a_\beta(f)\neq 0\,\,\forall n$. But Proposition \ref{prop-coefficients}
implies that $a_\beta$ is continuous on $A^p(\h_{m/n})$. Thus
$$\left|a_\beta \left(g_n-f\right)\right|\lesssim \|g_n-f\|_{A^p}\to 0\qquad\text{as }n\to \infty,$$ 
a contradiction.
\end{proof}

For $1<p<2$, the results are more complicated. In the first place, the sub-Bergman projections $\widetilde{\bm{B}^r}$ are only defined if $r\geq 2$; consequently no approximation theorem for the range $1<p<r<2$ follows from results in this paper. Additionally, the approximation result that does follow -- for the range $1<p< 2\leq r$ -- requires consideration of the partition \eqref{E:partition} in Proposition \ref{P:ApClasses}.

\begin{proposition}\label{T:ApproxGenHartogs2} Let $1<p<2$ and $p'$ be conjugate to $p$. In the partition \eqref{E:partition}, choose $k$ so that $p' < p_{k+1} = \frac{2m+2n}{-k}$.

Fix $r \in [\,p_k,p_{k+1})$. Then $f\in A^p(\h_{m/n})$ can be approximated by $A^r\left(\h_{m/n}\right)$ functions in the $L^p$ norm if and only if
$\widetilde{\bm{B}^r}f = f$.
\end{proposition}

\begin{proof}

Since $p' < p_{k+1}$, simple algebra shows that $q_{k+1} < p$, where $q_{k+1}$ is the conjugate exponent to $p_{k+1}$.  Since $p\in (q_{k+1},p_{k+1})$, Theorem \ref{T:TechnicalSubBergmanStatement} implies $\widetilde{\bm{B}^r}$ is bounded on $L^p$.

The rest of the proof is the same as for Proposition \ref{T:ApproxGenHartogs1}.
\end{proof}

\subsubsection{$L^2$-nearest approximant in $A^p$} Question (Q3) can be cast as a broad minimization problem. Suppose $\norm{\cdot}_X$ is an auxiliary norm on the space $L^p(\Omega)$, $\Omega\subset\C^n$ fixed.
\smallskip

\noindent {\it Problem:}  Given $g\in L^p(\Omega)$, find $G\in A^p(\Omega)$ so 
\begin{equation}\label{E:NormMinimizingFunction}
\norm{g-G}_X \le \norm{g-h}_X
\end{equation} 
for all $h\in A^p(\Omega)$.

For general $\norm{\cdot}_X$, techniques needed for this problem mostly await development.  But when $X=L^2(\Omega)$ the sub-Bergman operators give results.  Recall that for $p\ge2$, $\widetilde{\bm{B}^p}$ is the orthogonal projection from $L^2$ onto $G^{2,p}$, the latter space given in Proposition \ref{prop-subbergman1}.  If $\Omega$ is bounded, the diagram
\begin{center}
\begin{equation*}
\begin{matrix}[1.3]
L^p(\Omega) & \hookrightarrow & L^2(\Omega) \\
 \big\downarrow \bm{?} &  & \big\downarrow \widetilde{\bm{B}^p}\\
A^p(\Omega) & \hookrightarrow & G^{2,p}(\Omega) \\
\end{matrix}
\end{equation*}
\end{center}
summarizes relations between the function spaces, with $ \hookrightarrow$ denoting injection. Consider ``closest''  to mean closest measured by the $L^2$ norm in the following. If $g\in L^2(\Omega)$, the unique closest element in $G^{2,p}(\Omega)$ is $\widetilde{\bm{B}^p}g$. However when $\Omega =\h_{m/n}$, Theorem \ref{T:SubBergmanLpMapping} says that $\widetilde{\bm{B}^p}$ restricts to a bounded operator on $L^p(\h_{m/n})$.  It follows that $\widetilde{\bm{B}^p}g$ is also the closest element in 
$A^p(\Omega)$ to $g$. Thus,

\begin{proposition}\label{T:SubBergmanMinimization}
Let $p\ge2$ and $g\in L^p(\h_{m/n})$.  The function $ \widetilde{\bm{B}^p}g$ satisfies 
\begin{equation*}
\norm{g-\widetilde{\bm{B}^p}g}_{L^2} \le \norm{g-h}_{L^2}
\end{equation*}
for all $h\in A^p(\h_{m/n})$, with equality if and only if $h = \widetilde{\bm{B}^p}g$. 
\end{proposition}

\bigskip

\bibliographystyle{acm}
\bibliography{ChaEdhMcN18}

\end{document}